\documentclass[a4paper,reqno]{amsart}

\usepackage{amssymb,amscd,verbatim}
\usepackage{rotating,lscape,enumerate}
\usepackage[french,english]{babel}
\usepackage[T1]{fontenc}
\usepackage[latin1]{inputenc}
\usepackage{lmodern}
\usepackage{mathabx}
\usepackage{xspace}
\usepackage{setspace}

\usepackage[headings]{fullpage}
\usepackage[ddmmyyyy]{datetime}

\usepackage[dvipsnames]{xcolor}
\definecolor{officegreen}{rgb}{0.0, 0.5, 0.0}
\definecolor{electricultramarine}{rgb}{0.25, 0.0, 1.0}
 \usepackage[linktocpage=true,colorlinks=true,linkcolor=electricultramarine,urlcolor=violet,citecolor=officegreen,backref=page]{hyperref}

\newtheorem{thm}{Theorem}[section]
\newtheorem{lem}[thm]{Lemma}
\newtheorem{sublem}[thm]{Sublemma}
\newtheorem{prop}[thm]{Proposition}
\newtheorem{cor}[thm]{Corollary}
\newtheorem{remm}[thm]{Remark}

 \newtheorem{theorem}[thm]{Theorem}
\newtheorem{proposition}[thm]{Proposition}
\newtheorem{corollary}[thm]{Corollary}
\newtheorem{lemma}[thm]{Lemma}

\theoremstyle{definition}
\newtheorem{defn}[thm]{Definition}

\newtheorem{Notation}[thm]{Notation}

\newtheorem{rem}[thm]{Remark}
\newtheorem{rems}[thm]{Remarks}

\theoremstyle{remark}

\newtheorem*{Example}{Example}

\theoremstyle{plain}
\newtheorem{mainthm}{Theorem}

\numberwithin{equation}{section}

\usepackage[mathscr]{euscript}

\newcommand{\C}{{\mathbb{C}}}
\newcommand{\R}{{\mathbb{R}}}

\newcommand{\Z}{{\mathbb{Z}}}

\newcommand{\bbS}{{\mathbb{S}}}

\newcommand{\calX}{{\mathcal{X}}}

\newcommand{\calC}{{\mathcal{C}}}
\newcommand{\calO}{{\mathcal{O}}}

\DeclareMathAlphabet{\euls}{U}{eus}{m}{n}

\newcommand{\g}{{\mathfrak{g}}}  
 
\newcommand{\fm}{{\mathfrak{m}}}
\newcommand{\fn}{{\mathfrak{n}}}

\newcommand{\fk}{{\mathfrak{k}}}  

\newcommand{\fsl}{{\mathfrak{sl}}}  

\newcommand{\fz}{{\mathfrak{z}}}
\newcommand{\fl}{{\mathfrak{l}}}

\newcommand{\fs}{{\mathfrak{s}}}
 \newcommand{\su}{{\mathfrak{su}}}
\newcommand{\ft}{{\mathfrak{t}}}

\newcommand{\Csf}{{\mathsf{C}}}

\newcommand{\Qsf}{{\mathsf{Q}}}

\def\preisomto{\vbox{\hbox to
               14pt{\hfill$\sim$\hfill}\nointerlineskip\vskip -0.2pt
               \hbox to 14pt{\rightarrowfill}}}
\def\isomto{\mathop{\preisomto}}

\def\prelongisomto{\vbox{\hbox to
                17pt{\hfill$\sim$\hfill}\nointerlineskip\vskip -0.2pt
                \hbox to 17pt{\rightarrowfill}}}

\newcommand{\stoo}{\longrightarrow \kern-15pt
\longrightarrow}

\newcommand{\sto}{\twoheadrightarrow}
\newcommand{\ito}{\hookrightarrow}
\newcommand{\mto}{\mapsto}

\newcommand{\imply}{\Rightarrow}

\def\Ker{\operatorname {Ker}}

\def\End{\operatorname {End}}

\def\SL{\operatorname {SL}}

\def\SO{\operatorname {SO}}

\def\grad{\operatorname {grad}}

\def\Lie{\operatorname {Lie}}

\def\Unit{\operatorname {U}}

\def\Unit{\operatorname {U}}
\def\SU{\operatorname {SU}}

\def\diag{\operatorname {diag}}

\def\SO{\operatorname {SO}}

\DeclareMathOperator{\trace}{{trace}}
\DeclareMathOperator{\Ad}{{Ad}}

\def\boplus{\mathbin{\boldsymbol{\oplus}}}

\DeclareMathOperator{\rough}{{\nabla^*\nabla}}

\newcommand{\rn}{\mathbb R}

\newcommand{\f}{\mathcal{F}}

\newcommand{\tJ}{\tilde{J}}
\newcommand{\hh}{\mathcal{H}}
\newcommand{\vv}{\mathcal{V}}

\newcommand{\id}{{\mathrm{id}}}

\newcommand{\ascal}[2]{{\langle {#1}  \mid {#2} \rangle}}

\newcommand{\vpi}{\varpi}
\newcommand{\vepsilon}{\varepsilon}

\newcommand{\vtheta}{\vartheta}

\newcommand{\sminus}{{\smallsetminus}}

\newcommand{\gC}{{\mathfrak{g}_\C}}
\newcommand{\GC}{{G}_\C}
\newcommand{\groot}[1]{\g^{#1}_\C}
\newcommand{\grootR}[1]{\g^{#1}_\R}
\newcommand{\tC}{{\mathfrak{t}_\C}}
\newcommand{\tR}{{\mathfrak{t}_\R}}

\newcommand{\roots}{\mathsf{R}}
\newcommand{\proots}{\mathsf{R}^+}

\newcommand{\broots}{\mathsf{B}}
\newcommand{\rootsP}{\mathsf{P}}

\newcommand{\rootsQ}{\mathsf{Q}}
\newcommand{\prootsQ}{\mathsf{Q}^+}

 \newcommand{\wT}{{\widetilde{T}}}

\newcommand{\Kt}{{\widetilde{K}}}
\newcommand{\wK}{{\widetilde{K}}}

\newcommand{\wk}{{\widetilde{\mathfrak{k}}}}

\newcommand{\Ut}{{\widetilde{\mathcal{U}}}}
\newcommand{\wt}{{\widetilde{\mathfrak{t}}}}
\newcommand{\wz}{{\widetilde{\mathfrak{z}}}}

\newcommand{\wm}{{\widetilde{\mathfrak{m}}}}

\newcommand{\wpi}{{\widetilde{\pi}}}
\newcommand{\wpzero}{{\tilde{p}_0}}
\newcommand{\pzero}{{p_0}}
\newcommand{\ftkl}{\mathfrak{t}_{k,\ell}}
\newcommand{\Tkl}{{T}_{k,\ell}}

\newcommand{\ltr}{\mathtt{L}}
\newcommand{\rtr}{\mathtt{R}}
\newcommand{\half}{{\frac{1}{2}}}

\newcommand{\halfsqrt}{{\frac{1}{\sqrt{2}}}}
\newcommand{\kahler}{K\"ahler\xspace}
\newcommand{\sissi}{if and only if\xspace}
\newcommand{\ACS}{almost contact structure\xspace}
\newcommand{\NACS}{normal almost contact structure\xspace}
\newcommand{\NACMS}{normal almost contact metric
  structure\xspace}
\newcommand{\ACMS}{almost contact metric
  structure\xspace}
\newcommand{\txe}{{(\theta,\xi,\eta)}}

\newcommand{\txeM}{{(\theta,\xi,\eta,g_M})}
\newcommand{\txeMK}{{(\theta,\xi,\eta,g_{\scriptscriptstyle{G/\wK}})}}
\newcommand{\txeMT}{{(\theta,\xi,\eta,g_{\scriptscriptstyle{G/\wT}})}}
\newcommand{\reg}{{\mathrm{reg}}}

\begin{document}


\title[Normal almost contact metric
structures]{Harmonicity of normal almost contact metric
  structures} 

\author{M. Benyounes}
\author{T. Levasseur}
\author{E. Loubeau}
\address{Universit\'e de  Brest, CNRS UMR 6205, LMBA, F-29238 Brest, France}
\email{Michele.Benyounes@univ-brest.fr}
\email{Thierry.Levasseur@univ-brest.fr}
\email{Eric.Loubeau@univ-brest.fr}
\author{E. Vergara-Diaz}

\keywords{Harmonic section; harmonic map; harmonic unit vector
  field; normal almost contact metric structure; Riemannian submersion,
  flag manifold}
\subjclass[2010]{Primary: 53C15; Secondary: 53C43, 53D15, 58E20}
\thanks{Research carried out under the EC Marie Curie Action no. 219258 and supported by a grant of the Romanian Ministry of Research and Innovation, CCCDI-UEFISCDI, project number PN-III-P3.1-PM-RO-FR-2019-0234/1BM/2019, within PNCDI III., and the PHC Brancusi 2019 project no 43460 TL}

\begin{abstract}
  We consider normal almost contact structures on a Riemannian
  manifold and, through their associated sections of an ad-hoc
  twistor bundle, study their harmonicity, as sections or as
  maps. We rewrite these harmonicity equations in terms of the 
  curvature tensor and find conditions relating the
  harmonicity of the almost contact metric and almost complex structures
  of the total and base spaces of the Morimoto fibration. We
  apply these results to homogeneous principal circle bundles over generalised
  flag manifolds, in particular Aloff-Wallach spaces, to mass-produce harmonic almost contact metric structures.
\end{abstract}

\maketitle

\section{Introduction} \label{sec1}

Though there can be a wealth of almost contact structures on an odd-dimensional manifold, few tools exist to select a better one within this multitude. 

In presence of a Riemannian metric, compatible almost contact
structures can be interpreted as reductions of the orthonormal
frame bundle and, as such, are in one-one correspondence with
sections of an ad-hoc twistor bundle, constructed as the quotient
of the orthonormal frame bundle by the group
$\Unit(n) \times 1 \subset {\mathrm{SO}}(2n+1)$, just like in the
better known case of almost Hermitian structures and the group
$\Unit(n) \subset {\mathrm{SO}}(2n)$. Fitting this homogeneous
fibre bundle with a Sasaki-like metric allows one to unfold the
programme of harmonic map theory: construct a functional,
characterise critical points, study the associated flow,
determine stable maps. The objective is to compare almost contact
structures among themselves, so critical points must be
understood with respectto vertical variations.

The first steps were achieved in \cite{VW1} and an almost contact metric structure $(\theta, \xi,\eta)$ on a Riemannian manifold $(M^{2n+1},g)$ is harmonic if (we refer to Section 2 for notations and precise definitions) \cite[Theorem 3.2]{VW1}:

\begin{equation}\label{hse1}\tag{HSE1}
[\bar{\nabla}^{*}\bar{\nabla} \bar{J} , \bar{J}]=0 ,
\end{equation}
and
\begin{equation}\label{hse2}\tag{HSE2}
\nabla^{*}\nabla \xi = |\nabla\xi|^2 \xi- (1/2) \bar{J}\circ \trace(\bar{\nabla}\bar{J}\otimes\nabla\xi),
\end{equation}
where $\bar{J}$ (resp. $\bar{\nabla}$) is the restriction of $\theta$ to the horizontal distribution $\xi^\perp$ (resp. the Levi-Civita connection $\nabla$)  and $\nabla^{*}\nabla = -\trace \nabla^{2}$.

A supplementary condition appears when further requiring the
section to be a critical point of the energy for all possible
variations, i.e.,  a harmonic map \cite[Theorem 3.4]{VW1}:
\begin{equation}\label{hme}\tag{HME}
\sum_{i=1}^{2n+1}\Bigl(g(R(E_{i}, X) ,  \bar{J} (\overline{\nabla}_{E_i} \bar{J})) +
4g(R(E_{i},X)\xi , \nabla_{E_i}\xi)\Bigr) =0, 
\end{equation}
for any tangent vector field $X$ and an orthonormal frame $\{E_{i}\}_{i=1,\dots,2n+1}$.

Here a comparison should be drawn with its even-dimensional
counterpart, as the harmonicity equation for almost Hermitian
structures can be read as the limiting case of a Cartesian
product with a circle and a parallel Reeb vector field. Only
\eqref{hse1} remains and harmonic almost Hermitian structures are
described by the commutation of $J$ and $\nabla^{*}\nabla J$
\cite{CMW2}. A similar remark applies for the harmonic map
condition~\eqref{hme}.

Having the harmonic section equations~\eqref{hse1} and
\eqref{hse2}, the next task is to determine classes of almost
contact metric structures  which
are harmonic or, failing that, conditions so they become
harmonic. The first part was done in \cite{VW1,VW2} and
\cite{LV}, and this article investigates the particular case of
normal almost contact structures, that is almost contact metric structures whose cone construction is a Hermitian manifold.

After rewriting, in Theorems~\ref{thm3.1} and \ref{thm3.2},
the harmonic section equations for normal almost contact
structures in terms of the curvature tensor, echoing some results
in \cite{CMW2}, we apply, in Section~\ref{sec4} these properties
to the Morimoto fibration to obtain conditions linking the
harmonicity of the normal almost contact structure
$(\theta,\xi,\eta)$ on the total space of this circle bundle and
that of the Hermitian structure on the base.  They involve the
rough Laplacian of the Reeb vector field $\xi$ and the divergence of
the tensor $\theta$, 
cf.~Theorems \ref{thm4.1} and~\ref{thm4.1-harm}.   As a
consequence, we  get that   if these  rough Laplacian and divergence are
constant multiples of $\xi$, then the harmonicity of
$(\theta,\xi,\eta)$ is equivalent to that of the complex
structure on the base, see Corollary~\ref{cor3.9}.

This relationship then leads in Section~\ref{sec7} to an
extensive study of harmonic normal almost contact structures on
homogeneous principal circle bundles
over a generalised flag manifold. Specifically, let $G/K$ be a
generalised flag manifold obtained from a compact connected Lie
group~$G$.
If $G/\wK \sto G/K$ is a homogeneous principal $\bbS^1$-bundle, to each
Hermitian metric on $G/K$ one can associate  a (canonical) normal
almost contact structure on $G/\wK$. Applying Corollary~\ref{cor3.9},
one concludes 
that the almost contact metric structure is harmonic, as
a section or map, if and only if so is the Hermitian structure
on~$G/K$, see Theorems~\ref{prop6.16} and
\ref{thm6.17}. The construction of this almost contact metric structure
relies on the classical characterisation of invariant Hermitian
metrics on $G/K$ in terms of root systems; for the convenience of
the reader, we recall in Section~\ref{sec6} the results that we
need on the subject. 

Sections~\ref{sec6} and~\ref{sec7} lead to the construction of harmonic \NACS on $\bbS^1$-bundles over a generalised flag manifold $G/K$, summarised by the following statement (cf. Section~\ref{sec6} for notations and definitions).


\begin{mainthm}
  \label{thmA}
  Let $G$ be a compact Lie group, $G/K$ be a generalised flag
  manifold and denote by $C$ the centre of $K$.  Let
  $\rootsQ^+ \subset \roots$ be an invariant ordering of a root
  system $\roots$ and let  $R_T^+$ be the set of restrictions of the
  elements of $\rootsQ^+$ to $\fz = \Lie(C)$, so that the reductive
  decomposition associated to $K$ is:
    \[
    \textstyle{\g= \fk \boplus \fm, \quad \fm = \boplus_{\gamma
        \in R_T^+} \fm_\gamma,}
    \]
    where $\g=\Lie(G)$, $\fk = \Lie(K)$ and the $\fm_\gamma$'s are
    nonequivalent irreducible $K$-modules.  Consider the
    $G$-invariant Hermitian structure $(J,g_{\scriptscriptstyle{G/K}})$ on $G/K$ defined
    by:

    -- the $\Ad(K)$-invariant complex structure $J_\fm$
   on $\fm$  given by 
   $J_\fm X_\lambda= Y_\lambda$, $J_\fm Y_\lambda = - X_\lambda$,
   for all $\lambda \in \prootsQ$ (cf.~\eqref{eq1.1} for the
   definitions of $X_\lambda, Y_\lambda$).

 --  the $G$-invariant metric $g_{\scriptscriptstyle{G/K}}$ given by  $g_{\fm_\gamma} =
 \kappa_\gamma B$, $\kappa_\gamma >0$, on $\fm_\gamma$, where $B$ is an invariant
 scalar product on $\g$.

 Choose a sub-torus $\widetilde{C} \subset
 \fz$ of codimension one, a unit vector (for $B$) $X_0 \in
 \fz$ 
 such that $\fz$ is the orthogonal direct sum  $\R
 X_0 \oplus \Lie(\widetilde{C})$, and
 let $\wK$ be the closed connected subgroup of
 $K$ such that $\wk= \Lie(\wK)$ is the orthogonal of
 $X_0$ in $\fk$. Then $K/\wK \cong \bbS^1$ and $\pi : G/\wK \sto
 G/K$ is a $G$-homogeneous principal $\bbS^1$-bundle. 
 Define:

  -- the $G$-invariant principal connection
 $\eta(U) = B(X_0,U)A $ for all $U \in \g$; 

   -- the endomorphism $\theta = d\pi^{-1}\circ J\circ d\pi$;

    -- the metric  
   $g_{\scriptscriptstyle{G/\wK}} = \pi^* g_{\scriptscriptstyle{G/K}} + \eta \otimes \eta$;

   -- the vector field
   $\xi_p = \frac{d}{dt}_{\mid t = 0}(pe^{tX_0})$ for all
   $p \in G/\wK$.
   \\
   If $J$ is a harmonic section then the \NACS $(\theta,\xi,\eta)$ on
   $(G/\wK, g_{\scriptscriptstyle{G/\wK}})$ is a harmonic
   section. Moreover, when the metric
   $g_{\scriptscriptstyle{G/K}}$  is K\"ahlerian, i.e.
   $\kappa_\gamma + \kappa_{\gamma'} = \kappa_{\gamma + \gamma'}$
   for $\gamma, \gamma' \in R_T^ +$ such that
   $\gamma + \gamma' \in R_T^+$, the \NACS
   $(\theta,\xi,\eta)$ is a harmonic map.
 \end{mainthm}

As explained in Remark~\ref{rem6.18}, one can choose the constants $\kappa_\gamma$'s such that the harmonic \NACS is not Sasakian.

The article closes with a description of harmonic (map) normal
almost contact metric structures on the simplest significant
examples, namely the Aloff-Wallach spaces, which are principal
homogeneous $\bbS^1$-bundles over the full flag variety
$G/K=\SU(3)/T$.  Thus, $G/\wK= \SU(3)/T_{k,\ell}$ where
$T_{k,\ell}$ is a one dimensional sub-torus of the maximal torus
$T$ depending of a couple of integers $(k,l)\neq (0,0)$. As
explained in the previous theorem, we show
that the almost contact metric structure on
$\SU(3)/T_{k,\ell}$ associated to a K\"ahler metric on $\SU(3)/T$
is a harmonic map, and we give formulas for the normal almost
contact structures in terms of the integers $(k,\ell)$. In order to state
an explicit particular result, recall that we can take $T$ to be the subgroup of
diagonal matrices in $\SU(3)$, so that the elements of $\ft = \Lie(T)$ can be
written $i\diag[a,b,-(a+b)]$ with $a,b \in \R$ and
$\Lie(T_{(k,\ell)})= i\R\diag[k,\ell,-(k+\ell)]
$. We choose the 
scalar product on $\su(3)$ to be $B(U,V)= - \trace(UV)$. From
Theorem~\ref{thmA}  one deduces  the following
result (see Theorem~\ref{thm4.1a}).

\begin{mainthm}
  \label{thmB}
  Write $\su(3) = \ft \boplus \fm$ for the reductive
  decomposition associated to $\SU(3)/T$. Choose an invariant
  complex structure $J$ on $\SU(3)/T$, i.e.~a set
  $\proots = \{\alpha, \beta, \gamma =\alpha+\beta\}$ of positive
  roots for the root system defined by the pair $(\ft,\su(3))$.
  Then $\fm= \fm_\alpha \oplus  \fm_\beta \oplus \fm_\gamma$ as a $T$-module; 
  define the invariant K\"ahler-Einstein metric
  $g$  on $\SU(3)/T$ by $g= B$ on
  $\fm_\alpha \oplus \fm_\beta$ and $g= 2B$ on $\fm_\gamma$.
  \\
  For $k,\ell \in \Z$ such that $(k,\ell) \neq (0,0)$, set
  $X_0 = \frac{-i}{\sqrt{6(k^2 + \ell^2 +k \ell)}} \diag[2\ell+
  k, -\ell - 2k, - \ell +k]$. If $\pzero \in \SU(3)/T_{k,\ell}$ is the class of the
  identity, define $\theta, \xi, \eta$ by their values on
  $T_{\pzero}(\SU(3)/T_{k,\ell})$ identified with
  $\fm \boplus \R X_0$:
     \begin{gather*}
       \theta(X_\lambda) = Y_\lambda, \  \theta(Y_\lambda)  = -
       X_\lambda \ \; \text{for $\lambda \in \proots$},
       \ \; \theta(X_0) = 0,
       \\
       \eta(X_0)= 1, \ \; \eta(U) = 0 \ \; \text{for $U \in
         \fm$},  \quad \xi_{\pzero} = \frac{d}{dt}_{\mid t
         = 0}(\pzero e^{tX_0})= X_0.
     \end{gather*}
Set $g_{k,\ell} = \pi^*
g + \eta \otimes \eta$, hence
$g_{k,\ell}  = B$ on $\fm_\alpha
\boplus \fm_\beta \boplus \R X_0$ and
$g_{k,\ell} = 2B$ on $\fm_\gamma$. \\
Then $(\theta,\xi,\eta,g_{k,\ell})$ is a normal almost contact
metric structure on $\SU(3)/T_{k,\ell}$ and is a harmonic
map. One can choose $J$ such that  this structure is not Sasakian when $k \ne \ell$.
\end{mainthm}


Note that we adopt the 
following sign for the curvature
tensor: $R(X,Y)= [\nabla_{X},\nabla_{Y}] - \nabla_{[X,Y]}$.

We thank the anonymous referee for useful comments and remarks in the earlier version of this paper.

\section{Normal almost contact metric structures} \label{sec2}

An {\em almost contact structure} on an odd-dimensional
differentiable manifold $M^{2n+1}$ is a reduction of the
structure group of its tangent frame  to $\Unit(n) \times 1$. More
concretely, this is equivalent to the existence of a field of
endomorphisms $\theta$ of the tangent space, a vector field $\xi$
and a one-form $\eta$ related by
$$ \eta (\xi) =1 , \quad \theta^2 = - \mathrm{Id} + \eta\otimes \xi .$$
Then, necessarily $\theta \xi =0$ and $\eta\circ \theta =0$.
A Riemannian metric $g$ on $M^{2n+1}$ is {\em compatible} if it satisfies
$$ g(\theta X, \theta Y) = g(X,Y) - \eta(X)\eta(Y) ,$$
for all vector fields $X$ and $Y$ tangent to $M^{2n+1}$, and such a
metric can always be constructed from the data
$(\theta ,\xi,\eta)$. We will only consider almost contact
structures compatible with the Riemannian metric. This approach
is initially due to Gray~\cite{Gray} and a good treatment can be
found in~\cite{Blair}.

One can also lift the almost contact structure to the Cartesian
line-product $\tilde{M} = M^{2n+1}\times \rn$ and construct an
almost complex structure $\tJ$ by 
$$
\tJ(X +f\partial_t)= \theta X -f \xi +\eta (X)\partial_t,
$$ 
for $X\in \calX(M)$ and $\partial_t$ the canonical unit vector field
tangent to $\rn$.  When $\tJ$ is integrable, the almost contact
structure will be called {\em normal}, which can be characterised
by the equation
$$
N_{\theta} + 2\textrm{d}\eta \otimes \xi =0,
$$
where $N_{\theta}$ is the Nijenhuis tensor
$$N_{\theta}(X,Y)= \theta^2 [X,Y] + [\theta X,\theta Y] - \theta
[\theta X,Y] - \theta [X,\theta Y].$$ 

An immediate effect of this condition is that the vector field
$\xi$, while not necessarily Killing, must have  geodesic integral
lines and preserve the field of endomorphisms $\theta$,
i.e.~$\nabla_{\xi}\theta =0$ (cf.~\cite{Blair}).

To prepare the computations of the next section, we derive an
alternative characterisation of normal almost contact metric
structures. Proposition~\ref{prop1} is quite probably a well-known result but, with no reference available, a proof is included for the sake of completeness.

\begin{proposition}\label{prop1}
  Let $(M^{2n+1},g)$ be a Riemannian manifold equipped with an
  almost contact metric structure $(\theta,\xi,\eta)$.  Then this
  structure is normal if and only if
\begin{align*}
(\nabla_X \theta) (Y)  &= (\nabla_{\theta X} \theta) (\theta Y)
                         - \eta(Y)\nabla_{\theta X} \xi, 
\end{align*} 
for any vector fields $X$ and $Y$ on $M^{2n+1}$.
\end{proposition}

\begin{proof} 
  Let $\tilde{\nabla}$ denote the covariant derivative of
  $\tilde{M}$ and $\nabla$ the covariant derivative of $M$, on
  the factor $\rn$ derivation will be noted by $D$ and
  $\partial_t$ will be the unit vector field. Given two vector fields
  $X$ and $Y$ on $M$ and functions $F$ and $G$ of $t$, we
  consider the vector fields $(X+F\partial_t)$ and $(Y+G\partial_t)$
  tangent to $\tilde{M}=M\times \mathbb{R}$ and re-write the
  integrability condition of $\tJ$
  (cf. \cite[Prop. 7.1.3]{Baird-Wood})
\begin{equation}\label{eqint}
(\tilde{\nabla}_{\tJ(X+F\partial_t)} \tJ) (\tJ(Y+G\partial_t)) =(\tilde{\nabla}_{(X+F\partial_t)} \tJ) (Y+G\partial_t) .
\end{equation}
For the sake of simplicity, we will evaluate all expressions at a
given point $x\in M$ and assume that, around this point, we
locally extended the vector $Y\in T_{x}M$ such that $\nabla
Y(x)=0$.\\ 
On the one hand
\begin{align*}
&\left(\tilde{\nabla}_{\tJ(X,F\partial_t)} \tJ \right) \left(\tJ(Y+G\partial_t)\right)  \\
&= -\tilde{\nabla}_{(\theta X- F\xi+\eta (X)\partial_t)}(Y+G\partial_t) - \tJ (\tilde{\nabla}_{(\theta X- F\xi+\eta (X)\partial_t)}
(\theta Y - G\xi + \eta (Y) \partial_t) )\\
 &= -\nabla_{\theta X}Y + F \nabla_{\xi}Y -\theta(\nabla_{\theta X}\theta Y)+ G \theta(\nabla_{\theta X}\xi)+ F \theta(\nabla_{\xi}\theta Y)-F G \theta(\nabla_{\xi}\xi)\\
&+\left((\theta X)(\eta(Y))-F \xi((\eta(Y))\right)\xi
+\left(F \eta(\nabla_{\xi}\theta Y)- \eta((\nabla_{\theta X}\theta Y) \right)\partial_t,
\end{align*}
while, on the other hand, by definition of $\tJ$,
\begin{align*}
(\tilde{\nabla}_{(X+F\partial_t)} \tJ) (Y+G\partial_t) &= (\nabla_X \theta) (Y) - G \nabla_{X}\xi +g( \nabla_X\xi, Y) \partial_t.
\end{align*}
Equating these two computations by \eqref{eqint} and using
different values for $F$ and $G$, yields, for the $M$-component, 
\begin{align} 
(\nabla_X \theta) (Y) &= (\nabla_{\theta X} \theta) (\theta Y) - \eta(Y)\nabla_{\theta X} \xi \label{eqa1}\\
0&= \nabla_{\xi}Y+\theta(\nabla_{\xi}\theta Y)-\xi(\eta(Y))\xi
\label{eqa2}\\
- \nabla_{X}\xi &= \theta (\nabla_{\theta X}\xi) \label{eqa3}\\
\theta(\nabla_{\xi}\xi)&=0 ,\label{eqa4}
\end{align} 
and for the $\rn$-component
\begin{align} 
0= g(\nabla_{\xi}\theta Y, \xi), \label{eqa5}\\
g(\nabla_X\xi, Y)=-g(\nabla_{\theta X}\theta Y, \xi). \label{eqa6}
\end{align}
Choosing the right type of vector fields, it is easily proved that
Equation~\eqref{eqa1} implies \eqref{eqa2}, \eqref{eqa3} and
\eqref{eqa4}, and also Equations~\eqref{eqa5} and \eqref{eqa6},
so the normality of an almost contact metric structure is merely
equivalent to \eqref{eqa1}.
\end{proof}

From the proof of Proposition~\ref{prop1}, we easily obtain a
series of equations which  we will repeatedly use. 
\begin{corollary}\label{corollary1.1}
  Let $(M^{2n+1},g)$ be a Riemannian manifold equipped with a
  normal almost contact metric structure $(\theta,\xi,\eta)$,
  then we have
\begin{align}
g(\nabla_X \xi, Y) +\eta(\nabla_{\theta X}\theta Y)&=0, \label{eq1}\\
g(\nabla_{\xi}\theta Y, \xi)&=0,\label{eq2}\\
\nabla_{\xi} Y + \theta(\nabla_{\xi}\theta Y)- \xi\bigl(\eta(Y)\bigr) \xi &=0,\label{eq3}\\
\nabla_X \xi +\theta(\nabla_{\theta X}\xi)&=0,\label{eq4}\\
\nabla_{\xi} \xi&=0 \label{eq5}, 
\end{align}
for any vector fields $X$ and $Y$ on $M^{2n+1}$.
\end{corollary}

A consequence of Proposition~\ref{prop1} is that, on the
complement to the $\xi$-direction, the field of endomorphisms
$\theta$ behaves like an integrable complex structure. This could be seen as an almost contact version of \cite[Lemma 2.1]{VW2}.

\begin{corollary}\label{corollary1}
  Let $(M^{2n+1},g)$ be a Riemannian manifold equipped with a
  normal almost contact metric structure
  $(\theta,\xi,\eta)$. Denote by $\f$ the contact sub-bundle,
  i.e.~the distribution orthogonal to $\xi$, by $\bar{J}$ the
  restriction of $\theta$ to $\f$ and by $\bar{\nabla}$ the
  connection induced on $\f$ by $\nabla$. Then
$$(\bar{\nabla}_{\bar{J}X}\bar{J})(\bar{J}Y) = (\bar{\nabla}_X \bar{J})(Y),$$
for any $X$ and $Y$ in $\f$.
\end{corollary}

\section{Curvature equations of harmonicity} \label{sec3}

In the more favourable cases, the two equations characterising
the harmonicity of an almost contact metric structure can be expressed
in terms of the curvature tensor of $(M,g)$. This is what we
establish for normal almost contact metric structures in this section,
starting with a preliminary series of technical lemmas on a
Riemannian manifold $(M^{2n+1},g)$ equipped with a normal almost
contact metric structure $(\theta,\xi,\eta)$.

\begin{lemma}\label{lemma2}
  Let $E$ be a horizontal vector field, that is a section of the
  horizontal sub-bundle $\f$, such that, at some point
  $x\in M^{2n+1}$, $(\bar{\nabla}E)(x)=0$. Then, at $x$,
$$[E,\bar{J}E]= (\bar{\nabla}_E \bar{J}) (E) +2 g(E, \nabla_{\theta E} \xi) \xi .$$
\end{lemma}

\begin{proof}
If $E$ is a horizontal vector field then so is $\bar{J}E$ and, evaluating at $x$,
\begin{align*} 
[E,\bar{J}E] & = \bar{\nabla}_{E}(\bar{J}E)+g(\nabla_E (\bar{J}E), \xi) \xi -\bar{\nabla}_{\bar{J}E}E -g(\nabla_{\bar{J}E} E, \xi) \xi \\
& = (\bar{\nabla}_E \bar{J}) (E)+g(\nabla_E{\theta E}-\nabla_{\theta E}E, \xi) \xi\\
&= (\bar{\nabla}_E \bar{J}) (E)+g((\nabla_E\theta)(E)-\nabla_{\theta E}E, \xi) \xi.
\end{align*}
By Proposition~\ref{prop1}, 
$$g((\nabla_E\theta)(E), \xi)
=g((\nabla_{\theta E}\theta)(\theta E), \xi)
=- g(\nabla_{\theta E}E, \xi).$$
Then,
\begin{align*}
[E,\bar{J}E] & =(\bar{\nabla}_E \bar{J}) (E)-2g((\nabla_{\theta E}E, \xi) \xi
=(\bar{\nabla}_E \bar{J}) (E)+2g(E, \nabla_{\theta E}\xi) \xi,
\end{align*}
since $\xi$ is vertical.
\end{proof}

To compute the first harmonic section equation, we need to
express the commutator of $\bar{J}$ and its Laplacian.

\begin{lemma}\label{lemma3.2}
  Let $E$ be a section of $\f$, such that, at some point
  $x\in M^{2n+1}$, $(\bar{\nabla}E)(x)=0$. Then, at $x$,
  $$
  [\bar{\nabla}_E \bar{\nabla}_E \bar{J},\bar{J}] = -2 [\bar{R}(E,\bar{J}E),\bar{J}
  ]-2\bar{\nabla}_{(\bar{\nabla}_E \bar{J})(E)}\bar{J}- [\bar{\nabla}_{\bar{J}E}
  \bar{\nabla}_{\bar{J}E} \bar{J},\bar{J}],
  $$
where $\bar{R}$ is the curvature tensor of the contact sub-bundle
$\f$ equipped with the connection $\bar{\nabla}$.
\end{lemma}

\begin{proof}
  Let $E$ be a section of $\f$ and $X$ a horizontal vector 
  extended locally such that $\bar{\nabla} X =0$ at the point
  $x\in M$ where we evaluate all expressions.  By the Leibniz
  rule, we have
  $$
  (\bar{\nabla}_E\bar{\nabla}_E \bar{J})(\bar{J}X)=\bar{\nabla}_E
  \big((\bar{\nabla}_E \bar{J})
  (\bar{J}X)\big)-(\bar{\nabla}_{E}\bar{J})((\bar{\nabla}_E \bar{J}) (X)).
  $$
  Using Corollary~\ref{corollary1} and Lemma~\ref{lemma2}, the
  first term on the right-hand side may be expressed in terms of
  the curvature tensor $\bar{R}$ as follows
\begin{align*}
&\bar{\nabla}_E\big((\bar{\nabla}_E \bar{J}) (\bar{J}X)\big)  = -\bar{\nabla}_E\big((\bar{\nabla}_{\bar{J}E}\bar{J}) (X)\big) 
 =-\bar{\nabla}_E\bar{\nabla}_{\bar{J}E}(\bar{J}X)+\bar{\nabla}_E(\bar{J}\bar{\nabla}_{\bar{J}E}
  X) \\
  &= -\bar{\nabla}_E\bar{\nabla}_{\bar{J}E}(\bar{J}X)+ \bar{J}(\bar{\nabla}_{E}\bar{\nabla}_{\bar{J}E}X)\\ 
& =-\bar{\nabla}_{\bar{J}E}\bar{\nabla}_E(\bar{J}X)+\bar{J}\bar{\nabla}_{\bar{J}E}\bar{\nabla}_E X -\bar{\nabla}_{[E,\bar{J}E]}(\bar{J}X)+ \bar{J}\bar{\nabla}_{[E,\bar{J}E]}X-[\bar{R}(E,\bar{J}E),\bar{J}](X)\\
&= -[\bar{R}(E,\bar{J}E),\bar{J}](X)  -\bar{\nabla}_{\bar{J}E}\bar{\nabla}_E(\bar{J}X)+\bar{J}\bar{\nabla}_{\bar{J}E}\bar{\nabla}_E  X - \bar{\nabla}_{(\bar{\nabla}_E \bar{J})(E)}(\bar{J}X)
  - 2g(  E, \nabla_{\bar{J} E}\xi) \bar{\nabla}_{\xi} (\bar{J}X)
  \\ & \phantom{xxx}+\bar{J}(\bar{\nabla}_{(\bar{\nabla}_E \bar{J})(E)}X)
+2g(  E, \nabla_{\bar{J} E}\xi) \bar{J}(\bar{\nabla}_{\xi}X)\\
  & =  -[\bar{R}(E,\bar{J}E),\bar{J}](X)  -\bar{\nabla}_{\bar{J}E}\bar{\nabla}_E(\bar{J}X)+\bar{J}\bar{\nabla}_{\bar{J}E}\bar{\nabla}_E X - \big(\bar{\nabla}_{(\bar{\nabla}_E \bar{J})(E)}\bar{J}\big)(X),
\end{align*}
since $\bar{\nabla}_{\xi} \bar{J} =0$. Now 
\begin{align*}
-\bar{\nabla}_{\bar{J}E}\bar{\nabla}_E(\bar{J}X) +\bar{J} (\bar{\nabla}_{\bar{J}E}\bar{\nabla}_E X)
&= -\bar{\nabla}_{\bar{J}E}((\bar{\nabla}_E \bar{J})(X)) -\bar{\nabla}_{\bar{J}E}(\bar{J}(\bar{\nabla}_E X))+\bar{J} (\bar{\nabla}_{\bar{J}E}\bar{\nabla}_E X)\\
&= -\bar{\nabla}_{\bar{J}E}((\bar{\nabla}_E \bar{J})(X))
= - \bar{\nabla}_{\bar{J}E}((\bar{\nabla}_{\bar{J}E}\bar{J})(\bar{J}X))\\
&= - (\bar{\nabla}_{\bar{J}E}\bar{\nabla}_{\bar{J}E}\bar{J})(\bar{J}X)- (\bar{\nabla}_{\bar{J}E}\bar{J}) (\bar{\nabla}_{\bar{J}E}(\bar{J}X))\\
&= - (\bar{\nabla}_{\bar{J}E}\bar{\nabla}_{\bar{J}E}\bar{J})(\bar{J}X)- (\bar{\nabla}_{\bar{J}E}\bar{J})((\bar{\nabla}_{\bar{J}E}\bar{J})(X)),
\end{align*}
because of the way we choose to extend the vector $X$.
Therefore
\begin{align}  
\bar{\nabla}_E \big((\bar{\nabla}_E \bar{J}) (\bar{J}X)\big) &=-[\bar{R}(E,\bar{J}E),\bar{J}](X)-\big(\bar{\nabla}_{(\bar{\nabla}_E\bar{J})(E)}\bar{J}\big)(X)\label{clubsuit}\\
&-(\bar{\nabla}_{\bar{J}E}\bar{\nabla}_{\bar{J}E}\bar{J}) (\bar{J}X) - (\bar{\nabla}_{\bar{J}E}\bar{J})\circ(\bar{\nabla}_{\bar{J}E}\bar{J})(X)\notag.
\end{align}
By definition of the covariant derivative of $\bar{\nabla}_E \bar{J}$
$$
\bar{\nabla}_E ((\bar{\nabla}_E \bar{J}) (\bar{J}X))= (\bar{\nabla}_E
\bar{\nabla}_E \bar{J})(\bar{J}X)+(\bar{\nabla}_{E}\bar{J})(
(\bar{\nabla}_{E}\bar{J})(X)),
$$
hence, using Equation~\eqref{clubsuit}, we have
\begin{align}
(\bar{\nabla}_E \bar{\nabla}_E \bar{J})(\bar{J}X) 
&=  -(\bar{\nabla}_{E}\bar{J})((\bar{\nabla}_{E}\bar{J})(X))+\bar{\nabla}_E ((\bar{\nabla}_E \bar{J}) (\bar{J}X))\notag\\
&=  -(\bar{\nabla}_{E}\bar{J})\circ(\bar{\nabla}_{E}\bar{J})(X)-[\bar{R}(E,\bar{J}E),\bar{J}](X)-(\bar{\nabla}_{(\bar{\nabla}_E\bar{J})(E)}\bar{J})(X)\notag\\
&-(\bar{\nabla}_{\bar{J}E}\bar{\nabla}_{\bar{J}E}\bar{J}) (\bar{J}X)  - (\bar{\nabla}_{\bar{J}E}\bar{J})\circ(\bar{\nabla}_{\bar{J}E}\bar{J})(X)  \notag \\
&= -2(\bar{\nabla}_{E}\bar{J})\circ(\bar{\nabla}_{E}\bar{J})(X)-[\bar{R}(E,\bar{J}E),\bar{J}](X)-(\bar{\nabla}_{(\bar{\nabla}_E\bar{J})(E)}\bar{J})(X)\label{eqstar}\\
&-(\bar{\nabla}_{\bar{J}E}\bar{\nabla}_{\bar{J}E}\bar{J}) (\bar{J}X) ,\notag
\end{align}
since, by Corollary~\ref{corollary1},
$$(\bar{\nabla}_{\bar{J}E}\bar{J})\circ(\bar{\nabla}_{\bar{J}E}\bar{J})(X) =
(\bar{\nabla}_{E}\bar{J})\circ(\bar{\nabla}_{E}\bar{J})(X).$$ 
To compute the second term of $[\bar{\nabla}_E\bar{\nabla}_E \bar{J} ,
\bar{J}] (X)$ we proceed as follows: 
\begin{align*} 
(\bar{\nabla}_E \bar{\nabla}_E \bar{J})(\bar{J}^2X)&=-2\bar{J}\circ(\bar{\nabla}_{E}\bar{J})\circ(\bar{\nabla}_{E}\bar{J})(X)-[\bar{R}(E,\bar{J}E),\bar{J}](\bar{J}X)\\
&+\bar{J}(\bar{\nabla}_{(\bar{\nabla}_E\bar{J})(E)}\bar{J})(X)+(\bar{\nabla}_{\bar{J}E}\bar{\nabla}_{\bar{J}E}\bar{J}) (X),
\end{align*}
then
\begin{align*} 
-\bar{J}(\bar{\nabla}_E \bar{\nabla}_E \bar{J})(X)&=2(\bar{\nabla}_{E}\bar{J})\circ(\bar{\nabla}_{E}\bar{J})(X)-[\bar{R}(E,\bar{J}E),\bar{J}](X)\\
&-(\bar{\nabla}_{(\bar{\nabla}_E\bar{J})(E)}\bar{J})(X)+\bar{J}(\bar{\nabla}_{\bar{J}E}\bar{\nabla}_{\bar{J}E}\bar{J}) (X).
\end{align*}
Summing up this last equation with~\eqref{eqstar}, we obtain the result.
\end{proof}

\begin{lemma}\label{lemma3.3}
  Let $\{F_i\}_{i=1,\dots, 2n}$ be an orthonormal frame on the
  horizontal sub-bundle $\f$, such that, at the point
  $x\in M^{2n+1}$ where we evaluate all our expressions,
  $(\bar{\nabla} F_i)(x) =0$.  Define the Lee vector field by
  $\bar{\delta}\bar{J}=\sum_{i=1}^{2n}(\bar{\nabla}_{F_i}
  \bar{J})(F_i)$. Then
  $$
  \sum_{i=1}^{2n}\bar{\nabla}_{\bar{\nabla}_{\bar{J}F_i} (\bar{J}F_i)}\bar{J} =
  \bar{\nabla}_{\bar{J}\bar{\delta}\bar{J}}\bar{J} .
  $$
\end{lemma}

\begin{proof}
  Combining the Leibniz rule, Corollary~\ref{corollary1} and
  $\bar{\nabla} \bar{J} \circ \bar{J} =- \bar{J}\circ \bar{\nabla} \bar{J}$, yields
\begin{align*}
\sum_{i=1}^{2n}\bar{\nabla}_{\bar{J}F_i} (\bar{J}F_i) 
= \sum_{i=1}^{2n}(\bar{\nabla}_{\bar{J}F_i} \bar{J})(F_i) =- \sum_{i=1}^{2n}(\bar{\nabla}_{F_i} \bar{J})(\bar{J}F_i)= \bar{J}\bar{\delta}\bar{J} ,
\end{align*}
as required.
\end{proof}

We can now give the first harmonicity condition for a normal structure, reminiscent of \cite[Theorem 2.8]{CMW2}.

\begin{theorem}\label{thm3.1}
  Let $(M^{2n+1},g)$ be a Riemannian manifold equipped with a
  normal almost contact metric structure $(\theta,\xi,\eta)$.
  The first harmonic section equation~\eqref{hse1} for this
  structure can be written
  $$
  \sum_{i=1}^{2n}[\bar{R}(F_i,\bar{J}
  F_i),\bar{J}]=-2\bar{\nabla}_{\bar{\delta} \bar{J}} \bar{J},
  $$
where $\{F_i\}_{i=1,\dots, 2n}$ is an orthonormal frame on the
horizontal sub-bundle $\f$.
\end{theorem}

\begin{proof} Assume, without loss of generality, that the
  orthonormal frame $\{F_i\}_{i=1,\dots, 2n}$ satisfies
  $\bar{\nabla}F_i (x)=0$, for all $i$, at the point we evaluate
  at. First observe that
\begin{equation*}
-\bar{\nabla}^*\bar{\nabla}\bar{J} =\sum_{i=1}^{2n}(\bar{\nabla}_{F_i}\bar{\nabla}_{F_i}\bar{J} -\bar{\nabla}_{\nabla_{F_i}F_i}\bar{J}) +\bar{\nabla}_{\xi}\bar{\nabla}_{\xi}\bar{J} -\bar{\nabla}_{\nabla_{\xi}\xi}\bar{J} =\sum_{i=1}^{2n}\bar{\nabla}_{F_i}\bar{\nabla}_{F_i}\bar{J} 
\end{equation*}
since $\nabla F_i$ is in the $\xi$-direction, $\bar{\nabla}_{\xi}\bar{J}=0$ and $\nabla_{\xi}\xi =0$.
Working now with the orthonormal frame $\{\bar{J}F_i\}_{i=1,\dots,
  2n}$, 
$$-\bar{\nabla}^*\bar{\nabla}\bar{J}
  =\sum_{i=1}^{2n}\bar{\nabla}_{\bar{J}F_i}\bar{\nabla}_{\bar{J}F_i}\bar{J} -
    \bar{\nabla}_{\nabla_{\bar{J}F_i} (\bar{J}F_i)}\bar{J}
    -\bar{\nabla}^2_{\xi,\xi}\bar{J},$$ 
thus
\[
-\bar{\nabla}^*\bar{\nabla}\bar{J} =  \sum_{i=1}^{2n}\bar{\nabla}_{\bar{J}F_i}\bar{\nabla}_{\bar{J}F_i}\bar{J}  -
    \bar{\nabla}_{\bar{\nabla}_{\bar{J}F_i} (\bar{J}F_i)}\bar{J}  -g( \nabla_{\bar{J}F_i} (\bar{J}F_i), \xi ) \bar{\nabla}_{\xi} \bar{J} 
 =\sum_{i=1}^{2n}\bar{\nabla}_{\bar{J}F_i}\bar{\nabla}_{\bar{J}F_i}\bar{J}- \bar{\nabla}_{\bar{J}\bar{\delta} \bar{J}}\bar{J},
\]
by the previous lemma. Then,
$$-[\bar{\nabla}^*\bar{\nabla}\bar{J}
,\bar{J}](X)=\sum_{i=1}^{2n}[\bar{\nabla}_{\bar{J}F_i}\bar{\nabla}_{\bar{J}F_i}\bar{J},\bar{J}](X)
- [\bar{\nabla}_{\bar{J}\bar{\delta} \bar{J}}\bar{J},\bar{J} ](X).$$ 
Taking traces in Lemma~\ref{lemma3.2}, we get
\begin{align*}
-[\bar{\nabla}^*\bar{\nabla}\bar{J} , \bar{J}] (X)
=  -2\sum_{i=1}^{2n}[\bar{R}(F_i,\bar{J}F_i),\bar{J}](X) - 2 (\bar{\nabla}_{\bar{\delta}\bar{J}}\bar{J})(X)+[\bar{\nabla}^*\bar{\nabla} \bar{J} , \bar{J}](X) 
-[ \bar{\nabla}_{\bar{J}\bar{\delta} \bar{J}}\bar{J},\bar{J}](X).
\end{align*}
Notice that
$$[ \bar{\nabla}_{\bar{J}\bar{\delta} \bar{J}}\bar{J},\bar{J}](X) =
(\bar{\nabla}_{\bar{J}\bar{\delta}
  \bar{J}}\bar{J})(\bar{J}X)-\bar{J}(\bar{\nabla}_{\bar{J}\bar{\delta} \bar{J}}\bar{J})(X) 
= 2 (\bar{\nabla}_{\bar{\delta} \bar{J}}\bar{J})  (X),$$ 
by Corollary~\ref{corollary1}.
Therefore
$$[\bar{\nabla}^*\bar{\nabla}\bar{J} , \bar{J}]
(X)=\sum_{i=1}^{2n}[\bar{R}(F_i,\bar{J}F_i),\bar{J}](X)+ 2
(\bar{\nabla}_{\bar{\delta}\bar{J}}\bar{J})(X),$$
and the theorem follows.
\end{proof}

Recall 
that the rough Laplacian on $M$ is
given on $V \in \calX(M)$ by
\begin{equation} \label{laplacian}
\rough V=  - \trace_{g} \nabla^2 V=
  -\sum_{i=1}^{2n+1}  \nabla_{E_i}(\nabla_{E_i}V) -
  \nabla_{\nabla_{E_i}E_i}V
\end{equation}
where $\{E_i\}_{1 \le i \le 2n+1}$ is an orthonormal frame on $M$.  Moreover, if
$\phi \in \End TM$, let
\begin{equation} \label{deltatheta}
\delta\phi = \trace_{g} \nabla \phi = \sum_{i=1}^{2n+1}
  \nabla_{E_i} \phi(E_i) - \sum_{i=1}^m \phi(\nabla_{E_i} E_i).
\end{equation}

\begin{theorem} \label{thm3.2} Let $(M^{2n+1},g)$ be a Riemannian
  manifold equipped with a normal almost contact metric structure
  $(\theta,\xi,\eta)$. Then the second harmonic section
  equation~\eqref{hse2} for the structure $(\theta,\xi,\eta)$ is
  $$
  \nabla^{*}\nabla \xi - |\nabla\xi|^2\xi=
  (1/2)\sum_{i=1}^{2n}[R(F_i,\theta  F_i),\theta ] (\xi)+
  (\nabla_{\delta \theta} \theta) (\xi),
  $$
where $\{F_i\}_{i=1,\dots, 2n}$ is an orthonormal frame on the
horizontal sub-bundle $\f$ and $\delta\theta=\trace_g \nabla
\theta $. 
\end{theorem} 

\begin{proof} 
  Notice that
  $\delta\theta= \sum_{i=1}^{2n}(\nabla_{F_i}\theta)(F_i)$
  (since, by Proposition~\ref{prop1},
  $(\nabla_{\xi}\theta)(\xi)=0$).  Equation~\eqref{hse2} can be
  written
  $$
  (\nabla^*\nabla \xi)^{\mathcal{F}}=-\tfrac{1}{2}\sum_{i=1}^{2n}\bar{J}\circ (\bar{\nabla}_{F_i} \bar{J})(\nabla_{F_i} \xi)=
-\tfrac{1}{2}\sum_{i=1}^{2n}\theta\circ (\nabla_{F_i}
\theta)(\nabla_{F_i} \xi).
$$
Working with an orthonormal frame $\{F_i\}_{i=1,\dots, 2n}$ as in
Theorem~\ref{thm3.1}, then \label{page:9} 
\begin{align*}
\sum_{i=1}^{2n}(\nabla_{F_i} \theta)(\nabla_{F_i} \xi) & = \sum_{i=1}^{2n}\nabla_{F_i} (\theta (\nabla_{F_i} \xi)) -\theta (\nabla_{F_i}\nabla_{F_i} \xi)
 = \sum_{i=1}^{2n}\nabla_{F_i}\nabla_{\theta F_i}\xi + \theta (\nabla^* \nabla \xi)\\
&= \sum_{i=1}^{2n}R(F_i,\theta F_i)\xi +\nabla_{\theta F_i} \nabla_{F_i}\xi + \nabla_{[F_i,\theta F_i]}\xi + \theta (\nabla^* \nabla \xi)\\
&= \sum_{i=1}^{2n}R(F_i,\theta F_i)\xi -\nabla_{\theta F_i} (\theta \nabla_{\theta F_i}\xi) + \nabla_{[F_i,\theta F_i]}\xi + \theta (\nabla^* \nabla \xi)\\
&=\sum_{i=1}^{2n} R(F_i,\theta F_i)\xi -(\nabla_{\theta F_i}
  \theta) (\nabla_{\theta F_i}\xi) -\theta (\nabla_{\theta
  F_i}\nabla_{\theta F_i}\xi) +\nabla_{[F_i,\theta F_i]}\xi + \theta (\nabla^* \nabla \xi)\\
&= \sum_{i=1}^{2n}R(F_i,\theta F_i)\xi - (\nabla_{\theta F_i} \theta)(\nabla_{\theta F_i}\xi) -\theta (\nabla_{\nabla_{\theta F_i}(\theta F_i)}\xi) 
+\nabla_{[F_i,\theta F_i]}\xi + 2\theta (\nabla^* \nabla \xi).
\end{align*}
Notice that 
\begin{align}
\sum_{i=1}^{2n}\theta (\nabla_{\nabla_{\theta F_i}(\theta F_i)}\xi) &= \sum_{i=1}^{2n}\nabla_{\theta (\nabla_{\theta F_i}\theta) (F_i)}\xi 
= -\nabla_{\delta \theta} \xi , \label{eqt} 
\end{align} 
by Lemma~\ref{lemma3.3}, the geodesic integral curves of $\xi$
and  the general property (since $\theta^2=- \mathrm{Id} + \eta\otimes
\xi$): 
$$
(\nabla_X \theta)(\theta Y)+ \theta(\nabla_X \theta)(Y)=
g(\nabla_X \xi, Y) \xi + g(Y, \xi) \nabla_X \xi,
$$
for any $X$ and $Y$ in $\calX(M)$.
Similarly, with Lemma~\ref{lemma3.3}, we have
\begin{align*}
\sum_{i=1}^{2n}\nabla_{[F_i,\theta F_i]}\xi &= \nabla_{\bar{\delta}\bar{J} +2 \sum_{i=1}^{2n}g( F_i, \nabla_{\theta F_i}\xi) \xi} \xi = \nabla_{\bar{\delta}\bar{J}} \xi 
=\nabla_{\delta \theta} \xi .
\end{align*}
Thus
$$
\sum_{i=1}^{2n}(\nabla_{F_i} \theta) ( \nabla_{F_i} \xi) = \tfrac{1}{2}\sum_{i=1}^{2n}R(F_i,\theta F_i)\xi  +\theta (\nabla^* \nabla \xi) + \nabla_{\delta \theta} \xi
$$
and
$$\tfrac{1}{2}\sum_{i=1}^{2n}\theta \circ (\nabla_{F_i} \theta)
  (\nabla_{F_i} \xi)  = \tfrac{1}{4}\sum_{i=1}^{2n}\theta R(F_i,\theta F_i)\xi
-\tfrac{1}{2} (\nabla^* \nabla \xi)^{\mathcal{F}} +
\tfrac{1}{2}\theta \nabla_{\delta \theta} \xi.
$$
Then the second harmonic equation~\eqref{hse2} is
$$
(\nabla^* \nabla \xi)^{\mathcal{F}} =
-\tfrac{1}{2}\sum_{i=1}^{2n}\theta R(F_i,\theta F_i)\xi -\theta
\nabla_{\delta \theta } \xi.
$$
Since $-\theta (\nabla_{\delta \theta} \xi) = (\nabla_{\delta \theta}
\theta) (\xi)$ and
$-\theta R(F_i,\theta F_i)\xi = [R(F_i,\theta F_i), \theta ]
(\xi)$, the theorem follows.
\end{proof}

When the vector field $\xi$ happens to be harmonic, as a section
of the unit tangent bundle i.e.~$\nabla^{*}\nabla \xi = |\nabla\xi|^2\xi$ (cf.~\cite{Perrone}),
the two harmonic section equations merge into a single one.

\begin{corollary}\label{corollary3.1}
  Let $(M^{2n+1},g)$ be a Riemannian manifold equipped with a
  normal almost contact metric structure $(\theta,\xi,\eta)$. If
  the vector field $\xi$ is harmonic, as unit section, then the
  almost contact metric structure $(\theta,\xi,\eta)$ is a harmonic
  section if and only if
  $$
  \sum_{i=1}^{2n}[R(F_i,\theta  F_i),\theta ] =-2\nabla_{\delta
    \theta} \theta .
  $$
\end{corollary}

\begin{proof} 
If $\xi$ is harmonic then the second harmonic equation simplifies to
$$
0=\sum_{i=1}^{2n}[R(F_i,\theta  F_i),\theta ] (\xi)+
2(\nabla_{\delta \theta} \theta) (\xi).
$$ 
The link between the $\f$-component of the curvature tensor of
$(M^{2n+1},g)$ and the curvature of the connection $\bar{\nabla}$
was established in~\cite{thesis}:
$$
\bar{R}(X,Y)= R^{\mathcal{F}}(X,Y)+r(\nabla_X \xi , \nabla_Y
\xi),
$$
where $r$ denotes the curvature tensor of the unit sphere,
i.e.~$r(X,Y)Z = g(Y , Z) X - g(X , Z) Y$. 
If one proves that
\begin{equation}\label{eqr}
\sum_{i=1}^{2n}[r(\nabla_{F_i} \xi , \nabla_{\bar{J}F_i} \xi),\bar{J} ] =0 ,
\end{equation}
then the first harmonic section equation could be rewritten as
\begin{equation} \label{eq*}
\sum_{i=1}^{2n}[R(F_i,\theta  F_i),\theta ] -  g([R(F_i,\theta  F_i),\theta ], \xi) \xi=-2\nabla_{\delta \theta} \theta+ 2  g(\nabla_{\delta \theta} \theta ,\xi) \xi ,
\end{equation}
which, applied to $X\in \f$, gives the first harmonic section equation 
$$
\sum_{i=1}^{2n}[R^{\mathcal{F}}(F_i,\theta  F_i),\theta ] (X)=
-2(\bar{\nabla}_{\delta \bar{J}} \bar{J})(X) ,
$$ 
which is the $\f$-part of 
$
\sum_{i=1}^{2n}[R(F_i,\theta  F_i),\theta ] =-2\nabla_{\delta
  \theta} \theta$.
When evaluated on $\xi$, Equation~\eqref{eq*} becomes
$$
\sum_{i=1}^{2n}[R(F_i,\theta  F_i),\theta ] (\xi)-
g([R(F_i,\theta  F_i),\theta ](\xi), \xi) \xi=-2(\nabla_{\delta
  \theta} \theta) (\xi)+ 2   g((\nabla_{\delta \theta}
\theta)(\xi) ,\xi) \xi
$$
which is the second harmonic section equation
$$
\sum_{i=1}^{2n}[R(F_i,\theta  F_i),\theta ] (\xi)
=-2(\nabla_{\delta \theta} \theta) (\xi),
$$
i.e.~the $\xi$-component of 
$\sum_{i=1}^{2n}[R(F_i,\theta  F_i),\theta ] =-2\nabla_{\delta
  \theta} \theta$.
To show that~\eqref{eqr} is indeed valid, we will only need
Proposition~\ref{prop1} and the definition of the tensor $r$. Let
$V$ be a horizontal vector field, then
\begin{align*} 
r(\nabla_{F_i}\xi , \nabla_{\bar{J}F_i}\xi) (\bar{J} V) &=  g(\nabla_{\bar{J}F_i}\xi, \bar{J} V) \nabla_{F_i}\xi - g(\nabla_{F_i}\xi , \bar{J}V) \nabla_{\bar{J}F_i}\xi \\
& = -g(\nabla_{F_i}\xi  ,  V ) \bar{J}\nabla_{\bar{J}F_i}\xi +  g(\nabla_{\bar{J}F_i}\xi , V) \bar{J}\nabla_{F_i}\xi \\
& = \bar{J}\big(r(  (\nabla_{F_i}\xi) , \nabla_{\bar{J}F_i}\xi)V\big),
\end{align*}
hence the conclusion.
\end{proof}

\section{The Morimoto submersion} \label{sec4}

This section rewrites the equations of Theorems~\ref{thm3.1} and
\ref{thm3.2} for the total space of a circle bundle, in order to
relate the harmonicity of its normal almost contact metric
structure with the harmonicity of the Hermitian structure on the
base. To ensure the normality of the structure on the total
space, we use Morimoto's version of the Boothby-Wang fibration
and extend~\cite{VW2}.  A {\em Morimoto submersion} will be a
holomorphic Riemannian submersion
$\pi: (M^{2n+1},\theta, \xi, \eta, g_M)\longrightarrow
(N^{2n}, J, g_N)$, i.e.~with $d\pi (\theta X) = J d\pi (X)$,
from a normal almost contact metric manifold onto a Hermitian
manifold.  Conditions for a circle bundle
$\pi: M^{2n+1} \to N^{2n}$ to be a Morimoto fibration are given,
in \cite{Mor1}, when $N^{2n}$ is a complex manifold, in
\cite{Mor2}, when $M^{2n+1}$ carries a normal almost contact
metric structure.

These results will lead in subsequent sections to a wide-ranging construction of harmonic invariant normal almost contact metric structures on homogeneous principal circle bundles over generalised flag manifolds.


Denote by $R,\nabla$, resp.~$R^N,\nabla^N$,  the curvature and
the connection on $M$, resp.~$N$.
We recall formulas due to O'Neill~\cite{ONeill} in our context.
Let $\hh$ and $\vv$ be the projections onto the horizontal and
vertical parts of $TM$ and define the tensors $A$ and $T$
by
$$
A_{E} F = \vv \nabla_{\hh E} \hh F + \hh \nabla_{\hh E} \vv F \quad T_{E}F = \vv \nabla_{\vv E} \hh F + \hh \nabla_{\vv E} \vv F .
$$
However, since our vertical distribution is one-dimensional and
$\xi$ has geodesic integral curves and is self-parallel, the
tensor $T$ is identically zero.

For the Morimoto fibration, the curvature identities of
\cite{ONeill} (with opposite sign convention for the curvatures)
are, for $X,Y,Z$ and $W$ horizontal vector fields and $X_*,Y_*,Z_*$ and $W_*$ their images under $d\pi$:
\begin{align}\label{curvformul}
    g_M(R(X,Y)Z, W ) 
    = &  g_N( R^N(X_*,Y_*)Z_*, W_*) + 2 g_M(\nabla_{X}Y ,\xi)g_M(
      \nabla_{Z}W , \xi)  \\
  &-  g_M(\nabla_{Y}Z,\xi)
      g_M(\nabla_{X}W,\xi) 
    - g_M(\nabla_{Z}X, \xi) g_M(\nabla_{Y}W, \xi) \notag\\
    g_M(R(X,Y)Z, \xi ) 
    &= - Z (g_M(\nabla_{X} Y, \xi)) + g_M( \nabla_Z X, \nabla_{Y} \xi) - g_M( \nabla_Z Y, \nabla_X\xi) \notag\\
    g_M(R(X,\xi)Y, \xi ) 
&= -  g_M(\nabla_{\xi}\nabla_{X} Y,\xi) +  g_M(\nabla_{\xi} X , \nabla_{Y} \xi) -  g_M(\nabla_{\xi} Y, \nabla_{X}\xi) \notag
\end{align}

The curvature expressions of the previous section enable us to
link the harmonicity of the structures on the total space and the
base of the Morimoto fibration.  This completes a result
of~\cite{VW2}.

\begin{theorem} \label{thm4.1} Let
  $\pi: (M^{2n+1},\theta, \xi, \eta, g_M)\longrightarrow
  (N^{2n}, J, g_N)$ be a holomorphic Riemannian
  submersion, i.e.~with $d\pi (\theta X) =  J d\pi (X)$,
  from a normal almost contact metric manifold onto a Hermitian
  manifold. Then the almost contact metric structure is a harmonic
  section if and only if $ J$ is a harmonic section and
  $$
  2(\nabla^{*}\nabla \xi )^{\f} +2 \theta(\nabla_{\delta \theta} \xi)
  +\theta\grad \big( g_M(\delta \theta ,\xi)\big)=0.
  $$ 
\end{theorem}

\begin{proof}  
Recall from Theorem~\ref{thm3.1}, the first harmonic section
equation~\eqref{hse1}:  
$$
\sum_{i=1}^{2n}[\bar{R}(F_i,\bar{J}
F_i),\bar{J}]=-2\bar{\nabla}_{\bar{\delta} \bar{J}} \bar{J} .
$$

 \begin{sublem}\label{lemma4.1} 
We have
$d\pi(\bar{\nabla}_{\bar{\delta} \bar{J}} \bar{J})=\nabla^N_{\delta J
} J$.
\end{sublem}

\begin{proof}  
  Since $\pi$ is a Riemannian submersion, we have for any section
  $X, Y$ in $\f$:
\begin{align*}
0=(\nabla d\pi)(X, Y)=(\nabla^N_{d\pi(X)}\hat{Y})\circ \pi -d\pi(\nabla_X Y),
\end{align*}
where $d\pi(Y)=\hat{Y}\circ \pi$.
Then,
\begin{align*}
d\pi\left((\bar{\nabla}_{\bar{\delta} \bar{J}} \bar{J})(X)\right)
&=(\nabla^N_{d\pi(\bar{\delta}\bar{J})}J\hat{X})\circ \pi -
  J(\nabla^N_{d\pi({\bar{\delta}\bar{J}})}\hat X)\circ \pi
  =(\nabla^N_{\delta J}J\hat{X})\circ \pi -J(\nabla^N_{\delta
  J}\hat{X})\circ \pi\\ 
&=(\nabla^N_{\delta  J} J)(\hat{X})\circ \pi,
\end{align*}
using that $d\pi(\bar{\delta}\bar{J})=(\delta  J)\circ \pi$ and
$d\pi(\bar{J}X)=( J \hat{X})\circ \pi$ if
$d\pi(X)=\hat{X} \circ \pi$.
\end{proof}

Then, we concentrate on the left-hand side of the equation.  Let
$\{F_i\}_{i=1,\dots,2n}$ be a local orthonormal frame of the
distribution $\f$ which projects to an orthonormal frame
$\{(F_{i})_{*}\}_{i=1,\dots,2n}$ on $N$.  By the expression
of $\bar{R}$ given in the proof of Corollary~\ref{corollary3.1}
and Equation~\eqref{eqr},
$$
\sum_{i=1}^{2n}[\bar{R}(F_i,\bar{J} F_i),\bar{J}] =\sum_{i=1}^{2n}
[R^{\f}(F_i,\bar{J} F_i),\bar{J}].
$$
By the curvature equations of a submersion~\cite{ONeill} and for
$Z$ and $H$ in $\f$, we have
\begin{align*}
&-\sum_{i=1}^{2n}g_M( [R^{\f}(F_i,\bar{J} F_i),\bar{J}](Z) ,H ) =
  \sum_{i=1}^{2n}-g_M( R^{\f}(F_i,\bar{J} F_i)(\bar{J}Z) ,H) - g_M( R^{\f}(F_i,\bar{J}
  F_i)Z ,\bar{J}H) \\ & = \sum_{i=1}^{2n}- g_N( R^N((F_i)_*,  J (F_i)_*)  JZ_*,H_*) - \frac{1}{2} \eta([F_i, \theta F_i]) \eta([\theta Z, H]) +\frac{1}{4} \eta([\theta F_i, \theta Z])  \eta([ F_i, H]) \\ &+\frac{1}{4} \eta([\theta Z, F_i])  \eta([ \theta F_i, H])-  g_N(  R^N((F_i)_*,  J (F_i)_*) Z_*,  JH_*) 
- \frac{1}{2} \eta([F_i, \theta F_i])\eta([Z, \theta H])\\
&+\frac{1}{4} \eta([\theta F_i, Z])  \eta([ F_i, \theta H]) +\frac{1}{4} \eta([Z, F_i])  \eta([ \theta F_i, \theta H]).
\end{align*}
Then, by Equations~\eqref{eq1} and \eqref{eq4} of
Corollary~\ref{corollary1.1},
\[
\eta([X, \theta Y])=-\eta([\theta X, Y]) \ \; \text{for any section $X, Y$ in
$\f$,}
\]
and we  deduce
\begin{align*}
-\sum_{i=1}^{2n}g_M( [R^{\f}(F_i,\bar{J} F_i),\bar{J}](Z) ,H ) 
& =\sum_{i=1}^{2n}- g_N( R^N((F_i)_*,  J (F_i)_*)-  g_N(  R^N((F_i)_*,  J (F_i)_*) Z_*,  JH_*)\\
&=-\sum_{i=1}^{2n} g_N ( [R^N((F_i))_*, J (F_i)_*, J](Z_*) ,H_* ).
\end{align*}
Combining with Sublemma~\ref{lemma4.1}, since $\pi$ is a Riemannian
submersion, we obtain:
$$
\sum_{i=1}^{2n}[R^N((F_i)_*, J (F_i)_*), J]=-2
\nabla^N_{\delta J } J .
$$
which gives the first condition of \cite{CMW2}.\\

First observe that $\xi$ must be Killing: by the equations of a
submersion~\cite{ONeill} and for $X$ and $Y$ in $\f$, we have
\begin{align*}
  g_M(\nabla_X \xi, Y)=-g_M(\xi, \nabla_X Y)=-
  g_M\Bigl(\xi, \frac{1}{2}[X, Y]\Bigr)=g_M(\xi,
  \nabla_Y X)=- g_M(\nabla_Y \xi, X). 
\end{align*}
By Theorem~\ref{thm3.2}, the second harmonic section equation~\eqref{hse2} is
$$
2(\nabla^{*}\nabla \xi )^{\f}= \sum_{i=1}^{2n}[R(F_i,\theta
F_i),\theta ] (\xi)+ 2(\nabla_{\delta \theta} \theta) (\xi),
$$
but 
$[R(F_i,\theta  F_i),\theta ] (\xi) = -\theta R(F_i,\theta  F_i)\xi$,
and, if $Z\in \f$,
$$-g_M(\theta R(F_i,\theta  F_i)\xi , Z) = -g_M(R(F_i,\theta
F_i)\theta Z , \xi).
$$
Then by O'Neill's formulas~\eqref{curvformul} 
\begin{align*}
-g_M(R(F_i,\theta  F_i)\theta Z , \xi) 
&= (\theta Z)\big( g_M(\nabla_{F_i} (\theta F_i ),\xi)\big) - g_M(\nabla_{\theta Z}F_i  , \nabla_{\theta F_i} \xi) + 
g_M(\nabla_{\theta Z}(\theta F_i) , \nabla_{F_i}\xi).
\end{align*}
Now, by the choice of extension, $\bar{\nabla}_{\theta Z}F_i =0$
at the point we evaluate at, and
\begin{align*}
g_M(\nabla_{\theta Z}(\theta F_i) , \nabla_{F_i}\xi)
  &=-g_M((\nabla_{\theta Z}\theta) (\theta^2 F_i),\nabla_{F_i}\xi)
    = - g_M((\nabla_{Z}\theta) (\theta F_i) , \nabla_{F_i}\xi) \\ &
= - g_M((\nabla_{Z}\theta) (\theta^2 F_i) , \nabla_{\theta F_i}\xi) = g_M(\nabla_{Z}(\theta F_i) - \theta(\nabla_{Z}F_i), \nabla_{\theta F_i}\xi) \\
&= - g_M(\theta \nabla_{Z}(\theta F_i) , \nabla_{F_i}\xi)  ,
\end{align*}
but on the other hand
\begin{align*}
g_M(\nabla_{\theta Z}(\theta F_i) , \nabla_{F_i}\xi) &= - g_M((\nabla_{Z}\theta) (\theta F_i) , \nabla_{F_i}\xi) 
= g_M(\theta \nabla_{Z} (\theta F_i), \nabla_{F_i}\xi).
\end{align*}
So $g_M(\nabla_{\theta Z}(\theta F_i) , \nabla_{F_i}\xi) =0$ and
Equation~\eqref{hse2} reduces to 
$$
2(\nabla^{*}\nabla \xi )^{\f} +2 \theta(\nabla_{\delta \theta} \xi)
=\sum_{i,j=1}^{2n}(\theta F_j)\big(g_M(\nabla_{F_i} (\theta F_i
),\xi)\big) F_j.
$$
Since
\begin{equation*}
  \sum_{i=1}^{2n} g_M(\nabla_{F_i} (\theta F_i ),\xi) =  g_M((\nabla_{F_i} \theta)( F_i ) + \theta(\nabla_{F_i} F_i),\xi) 
 = \sum_{i=1}^{2n} g_M((\nabla_{F_i} \theta)( F_i ) ,\xi) =  g_M(\delta \theta ,\xi), 
\end{equation*}
we obtain:
\begin{align*}
  \sum_{i,j=1}^{2n}(\theta F_j)\big( g_M(\nabla_{F_i} (\theta F_i ),\xi)\big) F_j 
& = \sum_{j=1}^{2n}(\theta F_j)\big( g_M(\delta \theta ,\xi)\big)
  F_j
  = -\sum_{j=1}^{2n}( F_j)\big( g_M(\delta \theta ,\xi)\big) \theta F_j 
 \\ & = -\theta \grad \big(  g_M(\delta \theta ,\xi)\big).
\end{align*}
Then Equation~\eqref{hse2} is:
$$
2(\nabla^{*}\nabla \xi )^{\f} +2 \theta(\nabla_{\delta \theta} \xi)
+\theta\grad \big( g_M(\delta \theta ,\xi)\big)=0,
$$
as required.
\end{proof}

\begin{theorem} \label{thm4.1-harm} Let
  $\pi: (M^{2n+1},\theta, \xi, \eta, g_M)\longrightarrow
  (N^{2n}, J, g_N)$ be a holomorphic Riemannian
  submersion, i.e.~with $d\pi (\theta X) =  J d\pi (X)$,
  from a normal almost contact metric manifold onto a Hermitian
  manifold. If the almost contact metric structure $(\theta, \xi, \eta)$ is a harmonic
  section then it also defines a harmonic map if and
  only if $\xi(\vert \nabla \xi\vert^2)=0$, i.e. the bending of $\xi$ is constant along its curves (cf.~\cite{Wiegmink}), and, for all $X$ in
  $\f$:
  $$
  \sum_{i,j=1}^{2n} g_N(  R^N((F_i)_*, X_* )(F_j)_* ,  J(\nabla^N_{(F_i)_*}  J)(F_{j})_*) = 4 g_M(\nabla^* \nabla \xi ,  \nabla_{X}\xi),
$$
where $\{F_i\}_{i=1,\dots, 2n}$ is an orthonormal frame on the
horizontal sub-bundle $\f$.
\end{theorem}

\begin{proof}
A harmonic almost contact metric structure will also be a
harmonic map if 
\begin{equation}\label{eqhme}
\sum_{i=1}^{2n+1} \frac14 g_M(\bar{J} (\bar{\nabla}_{E_i}
\bar{J}) , R(E_{i}, X))  +  g_M(\nabla_{E_i}\xi , R(E_{i},X)\xi) =0,
\end{equation}
for any $X\in \calX(M)$, where $\{F_i\}_{i=1,\dots,2n}$ is an
orthonormal frame of the distribution $\f$ and  
$\{E_i\}_{1 \le i \le 2n+1}$ $=\{F_i\}_{1\le i\le 2n}\cup \{\xi\}$ a local orthonormal frame of $TM$.

First take $X\in \f$, extend it into a basic vector field in
$\f$, still called $X$, so Equation~\eqref{eqhme} reads: 
\begin{equation}\label{eqhmef}
\sum_{i=1}^{2n+1} \frac14 \sum_{j=1}^{2n} g_M(\bar{J} (\bar{\nabla}_{E_i} \bar{J})(F_j) , R(E_{i}, X)(F_j)) + g_M(\nabla_{E_i}\xi , R(E_{i},X)\xi) =0,
\end{equation}
and using O'Neill's formulas~\eqref{curvformul}, the first part
of \eqref{eqhmef} can be developed into 
\begin{align}
     g_M(\bar{J} (\bar{\nabla}_{E_i} \bar{J})(F_j) , R(E_{i}, X)(F_j)) &= g_M(\bar{J} (\bar{\nabla}_{F_i} \bar{J})(F_j) , R(F_{i}, X)(F_j)) +
    g_M(\bar{J} (\bar{\nabla}_{\xi} \bar{J})(F_j) , R(\xi, X)(F_j)) \notag \\
&= g_N(  R^N((F_i)_*, X_* )(F_j)_* ,
   J(\nabla^N_{(F_i)_*}  J)(F_{j})_*) \notag \\ & +2
  g_M(A_{F_i}X, A_{F_j} (\bar{J}(\bar{\nabla}_{F_i} \bar{J})(F_{j}))) \notag
   - g_M(A_{X}  F_j,A_{F_i}(\bar{J}(\bar{\nabla}_{F_i} \bar{J})(F_{j}))) \notag\\
  &-   g_M(A_{F_j}F_i,A_{X} (\bar{J}(\bar{\nabla}_{F_i} \bar{J})(F_{j}))) + g_M(R(\xi, X)(F_{j}) , \bar{J}(\bar{\nabla}_{\xi} \bar{J})(F_{j})) \notag\\
&= g_N(  R^N((F_i)_*, X_* )(F_j)_* ,  J(\nabla^N_{(F_i)_*}  J)(F_{j})_*) \label{e1}\\
&+2 g_M(\nabla_{F_i}X,\xi) g_M(\nabla_{F_j} (\bar{J}(\bar{\nabla}_{F_i} \bar{J})(F_{j})), \xi) \label{e2}\\
&- g_M(\nabla_{X}  F_j,\xi) g_M(\nabla_{F_i}(\bar{J}(\bar{\nabla}_{F_i} \bar{J})(F_{j})), \xi)\label{e3}\\
&- g_M(\nabla_{F_j}F_i,\xi) g_M(\nabla_{X} (\bar{J}(\bar{\nabla}_{F_i} \bar{J})(F_{j})), \xi) \label{e4}\\
&+g_M(R(\xi, X)(F_{j}) , \bar{J}(\bar{\nabla}_{\xi} \bar{J})(F_{j})) \label{e5}.
\end{align}
The term~\eqref{e5} is zero because $\nabla_{\xi} Y$ is
horizontal whenever $Y$ is horizontal, as $\nabla_{\xi}\xi =0$,
and $\nabla_{\xi} \theta =0$ by Corollary~\ref{corollary1.1}.\\ 
The term~\eqref{e2} vanishes because
\begin{align*}
    g_M(\nabla_{F_j} (\bar{J}(\bar{\nabla}_{F_i} \bar{J})(F_{j})), \xi)  & = -
     g_M(\bar{J}(\bar{\nabla}_{F_i} \bar{J})(F_{j}),  \nabla_{F_j} \xi)=
     g_M((\bar{\nabla}_{F_i} \bar{J})(F_{j}),  \theta \nabla_{F_j} \xi)
  \\ &= g_M((\bar{\nabla}_{F_i} \bar{J})(F_{j}),  \nabla_{\theta F_j} \xi)
    = -g_M((\bar{\nabla}_{F_i} \bar{J})(\theta F_{j}),  \nabla_{F_j}
       \xi) \\ &= -g_M((\bar{\nabla}_{F_i} \bar{J})(\bar{J} F_{j}),  \nabla_{F_j}
       \xi) = g_M(\bar{J} (\bar{\nabla}_{F_i} \bar{J})(F_{j}),  \nabla_{F_j} \xi).
\end{align*}
The sum of the terms~\eqref{e3} and \eqref{e4} is
\begin{align*}
 & - \sum_{i,j=1}^{2n}g_M(\nabla_{X}  F_j,\xi) g_M(\nabla_{F_i}(\bar{J}(\bar{\nabla}_{F_i} \bar{J})(F_{j})), \xi)
- \sum_{i,j=1}^{2n}g_M(\nabla_{F_j}F_i,\xi) g_M(\nabla_{X} (\bar{J}(\bar{\nabla}_{F_i} \bar{J})(F_{j})), \xi) \\
&=  - \sum_{i,j=1}^{2n}g_M(F_j, \nabla_{X}\xi) g_M(\bar{J}(\bar{\nabla}_{F_i} \bar{J})(F_{j}), \nabla_{F_i}\xi)
- g_M(F_i, \nabla_{F_j}\xi) g_M(\bar{J}(\bar{\nabla}_{F_i} \bar{J})(F_{j}), \nabla_{X}\xi) \\
&=  - \sum_{i=1}^{2n}g_M(\bar{J}(\bar{\nabla}_{F_i} \bar{J})(\nabla_{X}\xi), \nabla_{F_i}\xi)
+ g_M(\bar{J}(\bar{\nabla}_{F_i} \bar{J})(\nabla_{F_i}\xi), \nabla_{X}\xi) \\
&= 2\sum_{i=1}^{2n}g_M(\bar{J}(\bar{\nabla}_{F_i} \bar{J})(\nabla_{F_i}\xi), \nabla_{X}\xi) = -4 g_M(\nabla^* \nabla \xi ,  \nabla_{X}\xi),
\end{align*}
by \eqref{hse2}.

The second term in \eqref{hme} when $X\in \f$ is, using \eqref{curvformul}:
\begin{align*}
   &\sum_{i=1}^{2n}g_M(\nabla_{F_i}\xi , R(F_{i},X)\xi) +  g_M(\nabla_{\xi}\xi , R(\xi,X)\xi)
   = \sum_{i=1}^{2n}g_M(\nabla_{F_i}\xi , R(F_{i},X)\xi)  \\
   &= \sum_{i=1}^{2n}g_M(\nabla_{\nabla_{F_i}\xi} \Big( \nabla_{F_i} X \Big) ,\xi) + 
   g_M(\nabla_{F_i}X , \nabla_{\nabla_{F_i}\xi} \xi) +
    g_M(X , \nabla_{\nabla_{\nabla_{F_i}\xi} F_i} \xi) + 
     g_M(\nabla_{\nabla_{F_i}\xi}X ,\nabla_{F_i}\xi) \\
     &=  \sum_{i=1}^{2n} - g_M(\nabla_{F_i} X  , \nabla_{\nabla_{F_i}\xi} \xi)
     + g_M(\nabla_{F_i}X , \nabla_{\nabla_{F_i}\xi} \xi) =0.
\end{align*}
So the condition \eqref{hme} for $X\in \f$ is equivalent to:
$$
\sum_{i,j=1}^{2n} g_N(  R^N((F_i)_*, X_* )(F_j)_* ,  J(\nabla^N_{(F_i)_*}  J)(F_{j})_*) = 4 g_M(\nabla^* \nabla \xi ,  \nabla_{X}\xi).
$$

The first term in Equation~\eqref{eqhme} for the case $X=\xi$ is:
$$
\sum_{i=1}^{2n+1}g_M(\bar{J} (\bar{\nabla}_{E_i} \bar{J}) , R(E_{i}, \xi))
= \sum_{i,j=1}^{2n}g_M(\bar{J} (\bar{\nabla}_{F_i} \bar{J})(F_j) , R(F_{i}, \xi) F_j).
$$
Let $Y = \sum_{i,j=1}^{2n}\bar{J} (\bar{\nabla}_{F_i} \bar{J})(F_j)$ then $Y = \sum_{i,j=1}^{2n}\theta (\nabla_{F_i} \theta)(F_j)$.
Using O'Neill's formula or the fact that $\xi$ is a Killing
vector field, so $\nabla^{2}_{X,Y} \xi = -R(\xi, X)Y$, we have 
\begin{align*}
   & \sum_{i,j=1}^{2n}g_M(\bar{J} (\bar{\nabla}_{F_i} \bar{J})(F_j) , R(F_{i}, \xi) F_j) =
   \sum_{i,j=1}^{2n}g_M(\theta (\nabla_{F_i} \theta)(F_j) , \nabla_{F_i}\nabla_{F_j} \xi - \nabla_{\nabla_{F_i}F_j} \xi) \\
   &=
   \sum_{i,j=1}^{2n}g_M(\theta (\nabla_{F_i} \theta)(\theta F_j) , \nabla_{F_i}\nabla_{\theta F_j} \xi - \nabla_{\nabla_{F_i}\theta F_j} \xi) \\
   &=
   \sum_{i,j=1}^{2n}g_M(-\theta^2 (\nabla_{F_i} \theta)( F_j) , (\nabla_{F_i}\theta)(\nabla_{ F_j} \xi)
   + \theta(\nabla_{F_i}\nabla_{F_j} \xi)
   - \nabla_{(\nabla_{F_i}\theta)(F_j) + \theta (\nabla_{F_i} F_j)} \xi) \\
   &=
   \sum_{i,j=1}^{2n} g_M(-\theta^2 (\nabla_{F_i} \theta)( F_j) , (\nabla_{F_i}\theta)(\nabla_{ F_j} \xi))
   + g_M(-\theta^2 (\nabla_{F_i} \theta)( F_j) , \theta(\nabla_{F_i}\nabla_{F_j} \xi))\\
   &+ g_M(\theta^2 (\nabla_{F_i} \theta)( F_j) ,\nabla_{(\nabla_{F_i}\theta)(F_j)} \xi))
\end{align*}
and the last term is zero because $\xi$ is Killing.
It  follows that 
$$2 \sum_{i,j=1}^{2n}g_M(\theta (\nabla_{F_i} \theta)( F_j) , \nabla_{F_i}\nabla_{F_j} \xi)$$ 
is
equal to:
\begin{align*}
& \sum_{i,j=1}^{2n}g_M(-\theta^2 (\nabla_{F_i} \theta)( F_j) ,
  (\nabla_{F_i}\theta)(\nabla_{ F_j} \xi))
  =\sum_{i,j=1}^{2n}g_M(-\theta[\theta(\nabla_{F_i} \theta)( F_j)]
  , (\nabla_{F_i}\theta)(\nabla_{ F_j} \xi)) \\ &=\sum_{i,j=1}^{2n}g_M(-\theta[-(\nabla_{F_i} \theta)( \theta
  F_j)+g_M(\nabla_{F_i} \xi, F_j) \xi],
  (\nabla_{F_i}\theta)(\nabla_{ F_j} \xi))  =\sum_{i,j=1}^{2n}g_M(\theta(\nabla_{F_i} \theta)( \theta F_j), (\nabla_{F_i}\theta)(\nabla_{ F_j} \xi))\\
&=\sum_{i,j=1}^{2n} g_M(-(\nabla_{F_i} \theta)( \theta^2 F_j)+g_M(\nabla_{F_i}\xi,\theta { F_j}) \xi, (\nabla_{F_i}\theta)(\nabla_{ F_j} \xi))\\
&=\sum_{i,j=1}^{2n}g_M((\nabla_{F_i} \theta)(F_j),
  (\nabla_{F_i}\theta)(\nabla_{ F_j} \xi)) +\sum_{i,j=1}^{2n}g_M(\nabla_{F_i}\xi,\theta F_j) g_M(\xi, (\nabla_{F_i}\theta)(\nabla_{ F_j} \xi)).
\end{align*}
The first term on the right-hand side of the equation is:
\begin{align*}
\sum_{i,j=1}^{2n}g_M((\nabla_{F_i} \theta)(F_j), (\nabla_{F_i}\theta)(\nabla_{ F_j} \xi))
&=\sum_{i,j,k=1}^{2n}g_M((\nabla_{F_i} \theta)(F_j), (\nabla_{F_i}\theta)(g_M(\nabla_{ F_j} \xi, F_k) F_k))\\
&=\sum_{i,j,k=1}^{2n}g_M((\nabla_{F_i} \theta)(F_j), (\nabla_{F_i}\theta)(F_k)) g_M(\nabla_{ F_j} \xi, F_k)  =0.
\end{align*}
The second term on the right-hand side of the equation is:
\begin{align*}
&\sum_{i,j=1}^{2n}g_M(\nabla_{F_i}\xi,\theta F_j) g_M(\xi, (\nabla_{F_i}\theta)(\nabla_{ F_j} \xi))
=\sum_{i,j=1}^{2n}g_M(\nabla_{F_i}\xi,\theta F_j) g_M(\theta(\nabla_{F_i}\xi), \nabla_{ F_j} \xi) \\
&=-\sum_{i,j=1}^{2n}g_M(\nabla_{F_i}\xi,\theta F_j)g_M(\nabla_{F_i}\xi, \nabla_{ \theta F_j} \xi) =-\sum_{i,j=1}^{2n}g_M(\nabla_{F_i}\xi,
  \nabla_{\nabla_{F_i}\xi}\xi) =0.
\end{align*}
Then the first term of \eqref{hme} for $X=\xi$ is equal to zero.\\

As to the second term of the harmonic map equation when $X=\xi$,
we have, using the third and last of the
Formulas~\eqref{curvformul} 
\begin{align*}
&\sum_{i=1}^{2n+1}g_M(\nabla_{E_i}\xi , R(E_i , \xi)\xi) = -
  \sum_{i=1}^{2n}g_M(R(F_i , \xi)\nabla_{F_i}\xi, \xi) \\ 
&=  \sum_{i=1}^{2n}g_M(\nabla_{\xi}(\nabla_{F_i} \nabla_{F_i} \xi) , \xi)  - g_M(\nabla_{\xi} F_i , \nabla_{\nabla_{F_i}\xi} \xi) +  g_M(\nabla_{\xi}\nabla_{F_i} \xi , \nabla_{F_i} \xi) \\
&= \sum_{i=1}^{2n}\xi\Big( g_M(\nabla_{F_i} \nabla_{F_i} \xi , \xi) \Big)  + \frac12 \xi \Big (g_M(\nabla_{F_i}\xi ,\nabla_{F_i}\xi)\Big)= \xi\Big( g_M(\sum_{i=1}^{2n}\nabla_{F_i} \nabla_{F_i} \xi +  \frac12  \vert \nabla\xi\vert^2 \xi , \xi) \Big)  \\
&= - \frac12 \xi\Big( g_M(\nabla^* \nabla \xi , \xi)\Big) =  - \frac12 \xi\Big( \vert \nabla\xi\vert^2)\Big).
\end{align*}
So the condition \eqref{hme} for $X=\xi$ is equivalent to $\xi\Big( \vert \nabla\xi\vert^2 \Big)=0$.
\end{proof}

\begin{corollary}  \label{cor4.1}
Let $\pi: (M^{2n+1},\theta, \xi, \eta, g_M)\longrightarrow
(N^{2n}, J, g_N)$ be a holomorphic Riemannian submersion
from a normal almost contact metric  manifold onto a Hermitian manifold.  

If $\xi$ is a harmonic unit vector field then the almost contact metric
structure is a  harmonic section if and only if $ J$ is a
harmonic section and $2 \nabla_{\delta \theta} \xi 
+\grad \big( g_M(\delta \theta ,\xi)\big)=~0$.
Moreover $(\theta, \xi, \eta)$ defines a harmonic map if and only
if $ J$ is a harmonic map and $\xi(\vert \nabla
\xi\vert^2)=0$.

If $\xi$ is parallel, then the almost contact metric structure is a
harmonic section if and only if $J$ is a harmonic
section. 
Moreover $(\theta, \xi, \eta)$ defines a harmonic map if and only
if $ J$ is a harmonic map.
\end{corollary}



One can endow a principal $\bbS^1$-bundle $\pi : M \sto N$ over a
complex manifold $N$, equipped with a connection~$\eta$, of a
\NACS (see \cite{Mor1}). Moreover, from a
Hermitian metric on $N$ one deduces a \NACMS on the total space
$M$, cf.~\cite{Hat,Mor1,Ogi}. 
We may apply Theorem~\ref{thm4.1} and
Corollary~\ref{cor4.1} in this situation, as we now explain.

 Let
 $ \exp(\R A) = \bbS^1$ be a one-dimensional compact torus
 where $\fs= \Lie(\bbS^1) = \R A$. Let $\pi : M \sto N$ be a
 principal $\bbS^1$-bundle over a (connected) differentiable manifold
 $N$. Fix a principal connection $\eta : N \to \fs$; if we
 identify $\fs$ with $\R$ via $A \mto 1$, one can view $\eta$ as
 an $\bbS^1$-invariant $\R$-valued form on $N$. Since $\bbS^1$ is abelian,
 the curvature of $\eta$ is given
 by $d\eta(X,Y)= - \eta([X,Y])$ for $X,Y \in \calX(M)$. The
 connection $\eta$ yields a splitting ${TM= HM \boplus VM}$
 where $VM = \Ker d\pi$ is the vertical bundle and
 $HM= \Ker \eta$ is the horizontal bundle. Let $p \in M$, the
 inverse of the isomorphism $d\pi(p) : H_pM \isomto T_{\pi(p)}N$
 is denoted by $\psi(p)$.  
 The vector field
 induced by the (right) action of $\bbS^1$ on $M$ is denoted by
 $\xi$, that is to say
 $\xi_p= \frac{d}{dt}_{\mid t=0} p e^{t X_0}$ for all
 $p \in M$.  From \cite[\S2]{Hat} \cite[Theorem~6]{Mor1} and
 \cite[\S3]{Ogi} we get the following result.

 \begin{thm}
   \label{thm1.1}
   Let $J$ be an almost complex structure on $N$ and denote by
   $(g_N, J)$ an almost Hermitian metric. Set
   \begin{equation*} 
g_M = \pi^* g_N + \eta \otimes \eta,
\end{equation*}
and define $\theta \in \End TM$ by
\begin{equation}
  \label{eq3.1}
  \theta_p(X_p) = \psi(p)(Jd\pi(p).X_p) \ \, \text{for all $X \in
    \calX(M)$ and $p \in M$}.
\end{equation}
 Then, $\sigma= \txeM$ is an \ACMS on $M$ such that $\pi \circ
\theta= J \circ \pi$. The Reeb vector field $\xi$ is Killing for
$g_M$.
\\
Furthermore, the \ACS $\txe$ is normal \sissi $J$ is integrable and
$d\eta= \pi^* \Sigma$ for some $\Sigma \in \Omega^2(N,\R)$ such
that $\Sigma$ is $J$-invariant, i.e.~$\Sigma(JX,JY)= \Sigma(X,Y)$
for all $X,Y \in \calX(N)$.
 \end{thm}
 
 The \kahler form $\Omega^{g_N}$ associated
 to an almost Hermitian metric $(g_N,J)$ on $N$ is defined by
\[
\Omega^{g_N}(X,Y) = g_N(JX,Y).
\]
Let $\sigma=\txeM$ be a \NACMS on $M$ as above (hence $J$ is
integrable). One says that $\sigma$ is Sasakian if
$d\eta= \pi^* \Omega^{g_N}$, cf.~\cite[Definitions~2.2
\&~4.2]{Ogi}, and that $\sigma$ is c-Sasakian if there exists
$c >0$ such that $(\theta,\xi,\eta, cg_M)$ is Sasakian.  The
following corollary is a combination of results proved
in \cite[\S 4]{Ogi}.

\begin{cor}
  \label{cor3.8}
  Let $\pi : M \sto N$ be a principal $\bbS^1$-bundle over an almost
  Hermitian manifold $(N,J,g_N)$. Assume that the \ACMS  $\txeM$
  defined in Theorem~\ref{thm1.1} is normal.  Then, the
  following assertions are equivalent.
  \begin{enumerate}[\rm (i)]
  \item $\txeM$ is a c-Sasakian structure;
    \item $(g_N,J)$ is a \kahler metric and there exists $c >0$
      such that $\Sigma = c\, \Omega^{g_N}$, where $d\eta=
      \pi^*\Sigma$ as in Theorem~\ref{thm1.1}. 
  \end{enumerate}
\end{cor}

\begin{proof} It suffices to give the proof for $c=1$.
  \\
  (i) $\imply$ (ii): If $\txeM$ is Sasakian, we know that
  $d\eta= \pi^*\Omega^{g_N}= \pi^*\Sigma$, from which it follows
  $\Omega^{g_N}= \Sigma$ since
  $\Omega^{g_N}(X,Y)= \pi^*\Omega^{g_N}(X^\psi,Y^\psi)=
  \pi^*\Sigma(X^\psi,Y^\psi)= \Sigma(X,Y)$, where we denote by
  $X^\psi$ the horizontal lift of~$X$. Moreover, $(g_N,J)$ is
  \kahler by \cite[\S 4, Corollary]{Ogi}.
  \\
  (ii) $\imply$ (i): 
  Since $\pi^*\Omega^{g_N}= \pi^*\Sigma= d\eta$, $\txeM$ is
  almost Sasakian, see ~\cite[Definitions~4.2]{Ogi}, and  as
  $\txeM$ is normal, it is Sasakian.
\end{proof}

 Recall from the introduction 
 that $J$ being a harmonic section on the manifold
 $(N,g_N,J)$ is equivalent to $[(\nabla^N)^*\nabla^N J,J]= 0$ and
 $J$ is a harmonic map if, moreover, the following equation
\[
\sum_{i=1}^{2n} g_N(\nabla^N_{F_i} J, R^N(F_i,X)J) = 0 \ \,
\text{for all $X \in \calX(N)$}
\]
is satisfied. Here, 
$(F_i)_{1 \le i \le 2n}$ is an orthonormal frame
on $N$.
 Observe that if $(g_N,J)$ is a \kahler metric, i.e.~$\nabla^N
 J=0$, one has  $[(\nabla^N)^*\nabla^N J,J]= 0$ and $\nabla^N_{F_i}
 J=0$, thus $J$ is a harmonic map.

 \begin{cor}
   \label{cor3.9}
   Under the previous notation, assume that $\delta \theta$  and
   $\nabla^*\nabla \xi$  are
   constant multiples of $\xi$. Then: 

   {\rm (1)}  $\sigma$ is a harmonic section \sissi $J$ is a
   harmonic section;

   {\rm (2)}  $\sigma$ is a harmonic map \sissi $J$ is a
   harmonic map;
   
   {\rm (3)} if $(g_N,J)$ is a \kahler metric, $\sigma$ is a
   harmonic map.
 \end{cor}

 \begin{proof}
   By~\eqref{eq5} we have $\nabla_\xi \xi = 0$ (actually $\xi$
   is Killing vector field, cf.~Theorem~\ref{thm1.1}). Therefore
   $\nabla_{\delta \theta} \xi = 0$ and the function
   $g_M(\delta \theta,\xi)$ is constant, thus $ 2 \theta \nabla_{\delta \theta} \xi  + \theta
 \grad(g_M(\delta\theta,\xi)) =0$
   by Theorem~\ref{thm4.1} and (1) follows from
 Corollary~\ref{cor4.1}.
   \\
   It is known that if $\nabla^* \nabla \xi = f \xi$ for some
   function $f$, then $f= |\nabla{\xi}|^2$ (see, for example,
   \cite[Lemma~7.1, \& Proposition~8.1]{Str}). Thus
   $|\nabla{\xi}|^2$ is constant,
   $\xi\bigl(|\nabla{\xi}|^2\bigr) =0$ and
 Corollary~\ref{cor4.1}  yields~(2). Finally, (3) follows
   since $J$ is a harmonic map in the K\"ahlerian case.
    \end{proof}

In section~\ref{sec7} we will apply  Corollary~\ref{cor3.9} to  homogeneous principal
$\bbS^1$-bundles over generalised flag manifolds associated to a
compact semi-simple Lie group.  In the next section we fix the
notation used to described almost Hermitian structures on generalised flag manifolds.
    

    \section{Generalised flag manifolds} \label{sec6}

    Let $\roots$ be a
    reduced root system.  If we fix a basis $\broots$ of
    $\roots$, we denote by $\proots$ the set of positive roots.
    By \cite[Chapitre~8, \S~2 \& \S~4 Th\'eor\`eme~1]{Bou} one knows
    that there exists a split semi-simple real Lie algebra
    $\g_\R$ and a Cartan subalgebra $\tR$ of $\g_\R$ such that
    the root system of $\mathsf(\tR, \g_\R)$ is equal to
    $\roots$. We then have, for any choice of positive roots,
    $$\g_\R= \tR \boplus_{\alpha \in \proots} \bigl(\g_\R^{\alpha}
    \boplus \g_\R^{-\alpha} \bigr)$$
    where
    $$\g_\R^{\lambda} = \{v \in \g_\R : [h,v] = \lambda(h)v \ 
    \text{for all $h \in \tR$}\}$$ if $\lambda \in \roots$.
Set $\g_\C = \g_\R \otimes \C$, $\tC = \tR \otimes \C$,
$\groot{\lambda} = \g_\R^{\lambda} \otimes \C$, etc.,  then,
$\g_\C$ is a complex semi-simple Lie algebra 
We denote by $\GC$ a
connected complex semi-simple Lie group with Lie
algebra $\gC \ne 0$.

Denote by $\ascal{}{}$ a positive real scalar multiple of the
Killing form on $\gC$.  The restriction of $\ascal{}{}$ to $\tR$
defines a Euclidean scalar product on $\tR$,
see~\cite[Chapitre~8, \S~2, Remarque~2, p.~80]{Bou}.  Thus, for
each $\phi \in \mathfrak{t}^*_\R$ there exits a unique
$h_\phi \in \tR$ such that $ \ascal{h_\phi}{h} = \phi(h)$ for all
$h \in \tR$.  We have
$\tR= \boplus_{\beta \in \broots} \R h_\beta$ if $\broots$ is a
basis of $\roots$.  For each $\lambda \in \roots$ one can choose
a root vector $E_\lambda \in \grootR{\lambda}$ such that
$\grootR{\lambda} = \R E_\lambda$,
$[E_\lambda, E_{-\lambda}]= h_\lambda \in \tR$; hence,
$$[h_\lambda, E_{\mu}] = \mu(h_\lambda) E_\mu =
\ascal{h_\mu}{h_\lambda}E_\mu$$ for all $\mu \in \roots$.  One has
$\ascal{E_\lambda}{E_{-\lambda}} = 1$ and
$\ascal{E_\lambda}{E_{\mu}} = 0$ if $\mu \neq -\lambda$.

One defines a compact Cartan subalgebra by
\[
 \ft = i \tR = \bigoplus_{\beta \in \broots} \R ih_\beta.
\]
If $\fl \subset \ft$ is a subspace we set $\fl_\R= \{x \in \ft_\R
: ix \in \fl\}$, thus $\fl = i \fl_\R$.
Set, for $\lambda \in \roots$:
\begin{equation}
  \label{eq1.1}
  X_\lambda = \halfsqrt (E_\lambda - E_{-\lambda}),
  \quad Y_\lambda =\frac{i}{\sqrt{2}}
  (E_\lambda + E_{-\lambda}).
\end{equation}
Thus,  $[X_\lambda,Y_\lambda]  = ih_\lambda$ and  for all $H \in \tR$:
\begin{equation*}
  \label{eq1.2}
  [iH, X_\lambda] = \lambda(H) Y_\lambda, \quad   [iH,
  Y_\lambda] = -\lambda(H) X_\lambda. 
\end{equation*}
Notice that 
\begin{gather} \label{eq1.3a}
\ascal{X_\lambda}{Y_\lambda} =0, \ \; \ascal{X_\lambda}{X_\lambda} =
 \ascal{Y_\lambda}{Y_\lambda}= -1, \\
  \label{eq1.3}
 \ascal{Y_\lambda}{Y_\mu}= \ascal{X_\lambda}{X_\mu} =
  \ascal{X_\lambda}{Y_\mu}= 0 \ \; \text{for all $\lambda \neq
    \mu$ in $\proots$}.
\end{gather}
Set $\fm_\lambda =  \fm_{-\lambda} =\R X_\lambda \boplus \R
Y_\lambda$ and $\fn =
  \bigoplus_{\lambda \in \proots} \fm_\lambda$.
The $\ft$-modules $\fm_\lambda$ are pairwise non-isomorphic
irreducible $\ft$-modules. One has
$ [\ft, \fm_\lambda] \subset \fm_\lambda$,
$[\fm_\lambda,\fm_\lambda] = \R ih_\lambda$ and, if $\lambda
\neq \pm \mu$:
\begin{equation}
  \label{eq1.6a}
  [\fm_\lambda,\fm_\mu] \subset \fm_{\lambda+\mu} \boplus
  \fm_{\lambda-\mu}
 \end{equation}
 with the convention that $\fm_\nu = 0$ when $\nu \notin
 \roots$, see \cite[Lemma~9.2]{WG}.
Under this notation, the Lie subalgebra
\[
\g= \ft \boplus \fn
\]
is a compact real form of $\g_\C$.
We denote by $G$ and $T$ the closed connected Lie subgroups of
$G_\C$ with Lie algebras $\g$ and $\ft$. Hence, $G$ is compact
and $T$ is a maximal torus of~$G$.

One obtains a Euclidean scalar product on $\g$ by setting,
cf.~\cite[Chapitre~9, \S~1, p.~16]{Bou}:
\begin{equation*}
  \label{eq1.6}
  B(U,V)  = -  \ascal{U}{V} \ \; \text{for all $U,V \in \g$}.
\end{equation*}
If $\mathfrak{a} , \mathfrak{b},
\mathfrak{c}$ are subspaces  of $\g$ such that $\mathfrak{c} =
\mathfrak{a} \oplus \mathfrak{b}$ with $\mathfrak{a}$ orthogonal
to $\mathfrak{b}$ with respect to $B$ (i.e.~$\ascal{}{}$) we will write 
$\mathfrak{c} = \mathfrak{a} \obot \mathfrak{b}$.  In this case, we denote by
$x_{\mathfrak{a}}$ the orthogonal projection of $x \in \mathfrak{c}$
onto~$\mathfrak{a}$.

By~\eqref{eq1.3a}, $B(X_\lambda,X_\lambda) =
B(Y_\lambda,Y_\lambda) = 1$, $B(X_\lambda,Y_\lambda) =0$.
Therefore $\{X_\lambda,Y_\lambda\}$ is an orthonormal basis
(for $B$) of $\fm_\lambda$ and, by~\eqref{eq1.3},
$\bigl((X_\lambda,Y_\lambda)  \, ; \lambda \in \proots \bigr)$
is an orthonormal basis of $\fn$.  One verifies that
$\ascal{h_\alpha}{E_\lambda} = 0$ for all $\alpha \in \roots$, hence
$B(h_\alpha,X_\lambda) = B(h_\alpha,Y_\lambda) =0$ and
$B(\ft,\fn) =B(\fm_\lambda,\fm_\mu) =0$, if
$\lambda \neq \pm \mu$.

A \emph{generalised flag manifold} is a homogeneous space
$G/K$ where $K$ is the centraliser of a torus $C \subset G$;
we may assume that $C \subset T$ is the connected component of
the centre of $K$. When $C=T$, we obtain the full flag
manifold~$G/T$.  The group $K$ is compact and connected; we set
$\fk = \Lie(K)$, which is a reductive Lie algebra.  We mostly
follow \cite{Arv, WG} for notation and results on generalised
flag manifolds and we will assume that $C$ is non-trivial, i.e.~$K \ne G$.

Let $\fz =\Lie(C)$ be the centre of $\fk$. Since
$\ft \subset \fk \subset \g$ one can find a root subsystem
$\rootsP \subset \roots \subset \ft_\R^*$ such that,
setting $\fn^\rootsP =  {\sum_{\beta \in \rootsP}}
  \fm_\beta$ and $\ft^\rootsP =\sum_{\lambda \in \rootsP} \R
  \, ih_\lambda$, one has $\ft = \fz \boplus \ft^{\rootsP} \subset
  \fk = \ft \boplus
  \fn^\rootsP$.
Set:
\begin{equation*} \label{eq6.2}
  \rootsQ= \roots \sminus \rootsP, \qquad \fm = \sum_{\alpha \in
    \rootsQ} \fm_\alpha.
\end{equation*}
Then
$B(\fk,\fm) = 0$ and $\g= \fk \obot \fm$ is the
reductive decomposition of $\g$ associated to $K$, i.e.~to the
homogeneous space $G/K$.
\\
 Let $\varpi : \mathfrak{t}_\R^* \to \mathfrak{z}_\R^*$ be
the restriction map. The elements of  $R_T =
\varpi(\roots) = \varpi(\rootsQ)$ are called $T$-roots and 
parameterise the irreducible summands of
the $K$-module $\fm$ as follows, cf.~\cite[Theorem~7.3]{Arv}. If
$\gamma \in R_T \subset \fz_\R^*$, set:
\begin{equation} \label{eq6.4b}
  \varpi^{-1}(\gamma) = \{\alpha \in \rootsQ \, : \, \alpha_{\mid
  \fz} = \gamma\} \subset \rootsQ.
\end{equation}
The $\fk_\C$-module $\fm_\C = \fm \otimes \C$ decomposes as the
sum of the pairwise non-isomorphic irreducible modules
$\fm_\gamma^\C = \sum_{\alpha \in \vpi^{-1}(\gamma)} \C
E_\alpha$,
i.e.~$\fm_\C = \bigoplus_{\gamma \in R_T} \fm_\gamma^\C$.
Choose a set of positive roots
$\proots$ and define for $\gamma \in R_T$:
\begin{equation*}
  \label{eq6.4a}
  \fm_\gamma = \fm_{-\gamma} = (\fm_\gamma^\C \boplus \fm_{-\gamma}^\C)
  \cap \g = \bigoplus_{\alpha \in \varpi^{-1}(\gamma) \cap
    \proots} \fm_\alpha.
\end{equation*}
Then, the  $\fm_\gamma$  are
nonequivalent irreducible $K$-modules. In order to write $\sum_\gamma\fm_\gamma$ as a direct sum,  for each
class in $R_T/\{\pm 1\}$ we choose a representative element
$\gamma \in R_T$ and we denote by $R_T^+$ the set of these
representative elements. Observe that
\[
\rootsQ \cap \proots= \bigsqcup_{\gamma \in R_T^+}
(\vpi^{-1}(\gamma) \cap \proots), \quad \fm=  \bigoplus_{\gamma \in R_T^+} \fm_\gamma = \bigoplus_{\alpha
  \in \rootsQ \cap \proots} \fm_\alpha.
\]
Notice that  $\{(X_\alpha,Y_\alpha) \, : \, {\alpha \in
  \varpi^{-1}(\gamma) \cap \proots}\}$ is an orthonormal basis (for $B$) of
$\fm_\gamma$ and  that $B(\fm_\gamma,\fm_{\gamma'})= 0$ if $\gamma \ne
\gamma'$. 

We say that $J \in \End(\fm)$ is an invariant complex structure
on $\fm$ if $J^2 = -\id$ and $\Ad(k)J = J$ for all $k \in K$. We
extend $J$ to $\g$ by $J=0$ on $\fk$.  Recall that if $z \in \g$ one
defines the Kostant-Kirillov-Souriau (KKS for short) form
$\omega_{z}$ by
\begin{equation*} \label{KKZ}
\omega_z(A,A') = B(z,[A,A']) \ \; \text{for all $A,A' \in \g$.}
\end{equation*}
Notice that if $v \in \tR$ one has
$\omega_{iv}(X_\lambda,Y_\lambda) = \lambda(v)$ since
$B(iv,ih_\lambda) = \ascal{v}{h_\lambda} = \lambda(v)$.  The
proof of the next classical proposition is left to the reader.

\begin{prop}
  \label{prop6.6}
  Let $J$ be an invariant complex structure on $\fm$ and
  $z \in \fz$.
  \\
  {\rm (1)} For all $\gamma \in
  R_T$, there exists $\epsilon(\gamma)= \pm
  1$ such that if $\lambda \in \vpi^{-1}(\gamma)$ one has
  $J X_\lambda = \epsilon(\gamma) Y_\lambda$ and
  $J Y_\lambda= - \epsilon(\gamma) X _\lambda$.
  \\
  {\rm (2)} The form $\omega_{z}$ is $J
  $-invariant and $K$-invariant, i.e.,
\[
\text{$\omega_{z}(JA,JA') =
  \omega_{z}(A,A') = \omega_z(\Ad(k)A,\Ad(k)A')$ for all $A,A' \in
  \g$ and $k \in K$}. 
\]
Moreover, $\omega_z(A,A')= B(z,[A,A']_\ft) =
  B(z,[A,A']_\fz)$ and
\[
\omega_z(A,A') = \omega_z(A_\fm,A'_\fm) 
=\omega_z(JA_\fm,JA'_\fm).  
\]
{\rm (3)} Suppose that $z'  \in
  \fz$ and $\omega_z= \omega_{z'}$
  on $\fm \times \fm$, then $z=z'$.
 \end{prop}

The next two theorems summarise results proved in
 \cite[Propositions~7.5, 7.6, 7.7]{Arv} and \cite[Theorem~4.7,
 Proposition~9.3, Theorem~9.4]{WG}.  Let $e \in G$ be the identity
 and $\pzero = eK \in G/K$; recall that $T_\pzero (G/K)$ identifies
 with~$\fm$. 

\begin{thm}
  \label{thm6.7}
  {\rm (1)}
  The set of $G$-invariant almost complex structures $J$ on $G/K$
  is in bijection with the set of $ \Ad(K)$-invariant complex
  structure $J_\fm$ on $\fm$ or, equivalently, with the set of
  signs $\{\epsilon(\gamma) = \pm 1 : \gamma \in R_T^+\}$. The
  structure $J_\fm$ is defined by
  $J_\fm X_\lambda = \epsilon(\gamma) Y_\lambda$,
  $J_\fm Y_\lambda= - \epsilon(\gamma) X_\lambda$ for all
  $\lambda \in \varpi^{-1}(\gamma)\cap \proots$.
  \\
  {\rm (2)} Any $G$-invariant metric $g_{\scriptscriptstyle{G/K}}$ on $G/K$ is
  compatible with any invariant almost complex structure. It is determined
  by an $\Ad(K)$-invariant scalar product $g_\fm$ on $\fm$
  defined on each $\fm_\gamma$, $\gamma \in R_T^+$, by $g_{\fm_\gamma} =
  \kappa_\gamma B$ for some  $\kappa_\gamma >0$, hence,  $g_{\fm} (X_\lambda,X_\lambda) = g_{\fm} (Y_\lambda,Y_\lambda)
=\kappa_\gamma$  for all $\lambda  \in  \varpi^{-1}(\gamma)
  \cap \proots$.
\end{thm}

Set
$\fz_{\reg} = \fz_\R \sminus \bigl(\bigcup_{\gamma \in R_T}
\Ker(\gamma)\bigr)$ and define a Weyl chamber to be a connected
component of $\fz_\reg$.  The invariant orderings $\prootsQ$, as
defined in \cite[Chapter~7, \S 7]{Arv}, are in bijection with the
Weyl chambers in $\fz_\R$.  If $\mathsf{C}$ is a Weyl chamber,
pick $z \in \mathsf{C}$; then
$\prootsQ= \{\alpha \in \rootsQ \, : \, \alpha(z) >0\}$ gives an
invariant ordering. From $\prootsQ$ one deduces a set of positive
roots $\proots \supset \prootsQ = \proots \cap \rootsQ$. When an
invariant ordering is fixed, we will choose $\proots$ in this
way. Then, one may take $R_T^+=\varpi(\prootsQ)$ and if
$\gamma \in R_T^+$ one has, cf.~\eqref{eq6.4b} :
\[
  \varpi^{-1}(\gamma)= \varpi^{-1} (\gamma)\cap  \proots = \{\lambda \in \prootsQ : \varpi(\lambda) =
  \gamma\}.
\]
The $\Ad(K)$-invariant complex structure $J_\fm$ associated
to the invariant ordering is given by choosing the signs
$\epsilon(\gamma) = +1$ for $\gamma \in R_T^+$, hence:
\begin{equation} \label{Qstructure}
J_\fm X_\lambda= Y_\lambda, \quad J_\fm Y_\lambda = - X_\lambda, \quad
\text{for all $\lambda \in \prootsQ$}.
\end{equation}


\begin{thm}
  \label{thm6.8}
  Retain the notation of Theorem~\ref{thm6.7}.
  \\
  {\rm (1)} There is a  bijection  between the following sets:
  \begin{enumerate}[{\rm (i)}]
  \item $G$-invariant complex structures on $G/K$;
  \item $\Ad(K)$-invariant complex structures $J_\fm$ on $\fm$ as
    in~\eqref{Qstructure} defined by a choice of an invariant
    ordering $\prootsQ$, i.e.~a Weyl chamber $\Csf$.
 \end{enumerate}
{\rm (2)} Let $J$ be an invariant complex structure on $G/K$ and $g_{\scriptscriptstyle{G/K}}$
be a $G$-invariant metric on $G/K$.  The following are equivalent:
\begin{enumerate}[{\rm (i)}]
\item $(g_{\scriptscriptstyle{G/K}},J)$ is \kahler;
  \item for all $\gamma, \gamma' \in R_T^+=
\varpi(\prootsQ)$ such that $\gamma + \gamma'
    \in R_T^+$, one has $\kappa_\gamma + \kappa_{\gamma'} =
    \kappa_{\gamma + \gamma'}$;
    \item the \kahler form $\Omega_\fm(U,V)= g_{\fm}(J_\fm U,V)$
      is equal to the KKS form ${\omega_{ih}}_{\mid \fm\times \fm}$ for some $h$ in the
      Weyl chamber defined by $\prootsQ$.
      \end{enumerate}
      {\rm (3)} The metric $g_{\scriptscriptstyle{G/K}}$ on $G/K$ is \kahler-Einstein
      \sissi $\Omega_\fm = c {\omega_{ih_\rho}}$ where $\rho =
      \half \sum_{\alpha \in \prootsQ} \alpha$ and $c$ is
      a positive real constant.
\end{thm}

The list of generalised flag manifolds for classical groups can
be found in \cite[Chapter~7, \S 4, p.~100]{Arv}. \\
Examples illustrating the previous definitions and results are
given in \cite[Chapter~7,  \S\S~5 \& 8, pp.~102-103 \&
109-111]{Arv} for the full flag manifold $\SU(3)/T$ and the
generalised flag manifold $\SO(8)/(\Unit(2) \times \Unit(2))$.
(See also the example for  $\SU(3)/T$ given after Proposition~\ref{prop5.3}.)
Harmonic maps between generalised flag manifolds are studied in~\cite{Guest}.

\section{Homogeneous principal bundles over a generalised flag manifold}
\label{sec7}

Let $M=G/K$, $K= Z_G(C)$, be a generalised flag manifold  as in
Section~\ref{sec6}. 
Choose a sub-torus $\widetilde{C} \subset C$ of codimension
one. Set $\wz = \Lie(\widetilde{C})$ and pick a unit vector
$X_0 \in \fz$, i.e.~$B(X_0,X_0) =1$, such that $B(X_0,\wz) =0$,
hence $\fz= \R X_0 \obot \wz$. Then, if
$\wk= \wz \boplus [\fk,\fk]$ we get $\fk= \R X_0 \obot \wk$
and $\wk$ is the orthogonal of $X_0$ in $\fk$.
  One has $Y_0 = -i X_0 \in \fz_\R$
  and $ \wz  = \fz \cap
    \wk$.
  The algebra $\wk$ is reductive and
  $\wt = \wz \obot \ft^{\rootsP}$ is a Cartan subalgebra of
  $\wk$. 
  Let $\wK$ be the
  closed connected subgroup of $G$ such that $\Lie(\wK) = \wk$, 
   then $K/\wK \cong  C/\widetilde{C} \cong \bbS^1$ and 
we obtain a $G$-homogeneous principal $\bbS^1$-bundle:
  \begin{equation*}
    \label{eq.bundle}
    \pi :   M=G/\wK \sto N=G/K.
  \end{equation*}
  From $\g = \fk \boplus \fm = \wk \boplus \R X_0 \boplus \fm$ one
deduces that the $\wK$-invariant  decomposition associated to the
reductive homogeneous space $G/\wK$ is:
\[
  \g= \wk \obot \wm \ \; \text{with $\wm = \R X_0 \obot \fm$
    (hence $[\wk,\wm] \subset \wm$)}.
\]
Set $\wpzero= e\wK$. One can identify $\wm$ with $T_\wpzero (G/\Kt)$
and the differential
$d\pi(\wpzero) : T_\wpzero (G/\Kt) \sto T_\pzero (G/K)$  with
the orthogonal projection $\wm \sto \fm$. Recall that there is a
bijection between $G$-invariant tensors on $G/\Kt$ and
$\Ad(\wK)$-invariant tensors on $\wm$, cf.~\cite{CS}. We will
emphasise this bijection in a few cases which enable us to
construct invariant structures on~$G/\Kt$.

Let $g \in G$, we denote by $\ltr_g : G/\Kt \to G/\Kt$, or
$\ltr_g : G/K \to G/K$, the left translation by $g$,
i.e.~$x \mto gx$. If $\calO \subset G/\Kt$ is open and
$Z \in \calX(\calO)$, the vector field
$g_*Z \in \calX(g\calO)$ is defined by
\[
g_*Z_{gq} = d\ltr_g(q). Z_q \ \; \text{ for all $q \in \calO$.}
\]
A vector field  $V \in \calX(G/\Kt)$ is said to be left-invariant 
if  $g_* V= V$ for all $g \in G$, that is $V_{gq} =
d\ltr_g(q).V_q$ for all $g \in G$ and $q \in G/\Kt$.
We denote by $\calX(G/\Kt)^G$ the vector space of left-invariant
vector fields.
Suppose that $V \in \wm$ centralises $\wk$.  The one-parameter
subgroup $\exp(\R V)$ acts by right translation on $G/\Kt$,
and one defines the vector field $V ^\#$ on $G/\Kt$ by:
\begin{equation*}
  \label{34}
  V ^\#_q= \frac{d}{dt}_{\big| t= 0} \left(qe^{tV} \right) \ \;
  \text{for $q \in G/\Kt$.}
\end{equation*}
It is easy to see that $V^\#$ is left-invariant.  The next 
(well-known) lemma recalls that we can obtain $\calX(G/\Kt)^G$
through that construction.

\begin{lem}
  \label{lem2.1}
  There exists a linear bijection $f$ between  $\calX(G/\Kt)^G$
  and
  $\wm^\wK= \{V  \in \wm \, : \, \Ad(h)V= V \ \, \text{for all $h
    \in \wK$}\} =  \{V  \in \wm \, : \, [\wk,V]= 0\}$.
 One has  $f(X) =X_{\wpzero}$ and $f^{-1}(V) = V^\#$.
\end{lem}

Using the previous lemma one can prove the following
proposition. Since we will not use this result, the proof is left
to the interested reader.  

\begin{prop}
  \label{prop6.9}
  The space $\calX(G/\Kt)^G$ is equal to $\R X_0^\#$ unless
  $\wt = i \Ker(\alpha)$ for an $\alpha \in \rootsQ$ such that,
  setting $\delta = \vpi(\alpha)$, the $\wk_\C$-module
  $\fm_\delta^\C = \C E_\alpha$ is of dimension one.  In this
  case one has
  $\calX(G/\Kt)^G = \R X_0^\# \boplus \R X_\alpha^\# \boplus \R
  Y_\alpha^\#$ and
  $ih_\alpha = \pm \sqrt{\ascal{h_\alpha}{h_\alpha}} X_0 \in
  \fz$. 
\end{prop}

\begin{defn} \label{rem6.9}
Set $\xi = X_0^\#$.
We say that $\wk$ is \emph{generic} if $\wm^\wk= \R X_0$,
i.e.~$\calX(G/\Kt)^G= \R \xi$. 
\end{defn}

\noindent Notice that, by the previous proposition, $\wk$
is generic except in a finite number of cases.

We now study some invariant vector fields which are important in
the construction of harmonic normal almost contact metric
structures on~$G/\Kt$, cf.~\S\ref{thm1.1}. 

Endow $G/\Kt$ with a $G$-invariant metric $g_{\scriptscriptstyle{G/\wK}}$, i.e.:
\[
(g_{\scriptscriptstyle{G/\wK}})_{kp}(d\ltr_k(p).U_{p},d\ltr_k(p).V_{p}) =
(g_{\scriptscriptstyle{G/\wK}})_{k}(U_{p},V_{p}) 
\]
for all $U,V \in \calX(G/\Kt)$, $p \in G/\Kt$ and $k \in
G$. Since $G/\Kt$ is a reductive homogeneous space, $g_{\scriptscriptstyle{G/\wK}}$ is
uniquely determined by an $\Ad(\wK)$-invariant Euclidean product
$g_\wm$ on $\wm$. Denote by $\nabla$ the Levi-Civita connection
associated to $g_{\scriptscriptstyle{G/\wK}}$. Then $\nabla$ is $G$-invariant, that is to
say:
\begin{equation*} \label{eq2.1c}
 g_* \nabla_X Y = \nabla_{g_*X} g_*Y  \ \; \text{for all $g\in G$
   and $X,Y \in \calX(G/\Kt)$.}
\end{equation*}
There is a bijection between the space of $G$-invariant
  connections on $G/\Kt$ and the space of bilinear maps $\Lambda:
  \wm \times \wm \to \wm$ which are $\wK$-invariant,
  i.e.~$\Lambda(\Ad(h)U,\Ad(h)V) =\Lambda(U,V)$ for $U,V \in \wm$ and $h
  \in \wK$ (see \cite{Nom}
  and \cite{KN} for the general theory).
We quickly recall, for further use, the construction of a connection $\nabla$
starting from~$\Lambda$ as given in~\cite[\S 7]{Nom}.  Firstly, if
$\wpi : G \sto G/\wK$ is the natural projection, there exists $\mathcal{U}$ contained in a neighbourhood of
$e \in G$ such that
$\wpi : \mathcal{U} \to \Ut= \wpi(\mathcal{U})$ is a diffeomorphism. 
Then, for $X \in \wm$ one can define $X^* \in \calX(\Ut)$ by
setting $X^*_{k\wpzero} = d\ltr_k(\wpzero).X$ for all $k \in \mathcal{U}$.
The connection $\nabla$ is completely determined on $\Ut$ by
$\nabla_{X^*}Y^* = \Lambda(X,Y)^*$ for $X,Y \in \wm$. 
This local construction then furnishes the invariant connection $\nabla$ on
the manifold~$G/\Kt$. In particular, if $\calO$ is an open set 
of $G/\Kt$, one has on the translated open set $g\calO$:
\begin{equation}
  \label{eq2.1b}
 g_*\nabla_XY =
\nabla_{g_*X}g_*Y \ \; \, \text{for all $X, Y \in \calX(\calO)$}. 
\end{equation}
By \cite[Theorem~13.1]{Nom} and
\cite[Theorem~3.3 \& p.~202]{KN} the map $\Lambda$ corresponding to the Levi-Civita connection
$\nabla$ associated to the chosen invariant metric $g_{\scriptscriptstyle{G/\wK}}$ is given by the
following formula for $X,Y \in \wm$:
\begin{equation}
  \label{eq2.2}
\Lambda(X,Y) = \half [X,Y]_\wm + U(X,Y)   
\end{equation}
where $U(X,Y) \in \wm$ is determined by
\begin{equation}
  \label{eq2.3}
 2g_\wm(U(X,Y),Z) = g_\wm(X,[Z,Y]_\wm) +
   g_\wm(Y,[Z,X]_\wm) \  \; \text{for $Z \in \wm$.}
\end{equation}

Let $\vtheta \in \End T(G/\Kt)$; one says that $\vtheta$ is
$G$-invariant if:
\begin{equation*}
  d\ltr_g(q) \circ \vtheta_q= \vtheta_{gq} \circ
  d\ltr_g(q) \ \; \text{for all $g \in G$ and $q \in G/\Kt$}.
\end{equation*}
Recall, cf.~\eqref{deltatheta}, that
$\delta \vtheta \in \calX(G/\Kt)$ is defined by
$$\delta \vtheta = \sum_{i=1}^m \nabla_{E_i} \vtheta(E_i) -
\sum_{i=1}^m \vtheta(\nabla_{E_i} E_i),$$
where $(E_i)_{1 \le i \le m}$ is an orthonormal frame
on a open subset $\calO$ of $G/\Kt$.
Recall also, see~\eqref{laplacian}, that for $V \in \calX(G/\Kt)$ we have defined $\rough
V \in \calX(G/\Kt)$ by 
$$\nabla^*\nabla V=-\sum_{i=1}^m
\nabla_{E_i}(\nabla_{E_i}V) - \nabla_{\nabla_{E_i}E_i}V.$$

We will use the following remark.

\begin{remm} \label{rema}
   If $(E_i)_{1 \le i \le m}$ is an orthonormal frame on an open
subset $\calO$ of $G/\Kt$ and $g \in G$, then
$(g_*E_i)_{1 \le i \le m}$ is an orthonormal frame on the open
subset $g\calO$.
\end{remm}

\noindent  The remark follows from the invariance
of $g_{\scriptscriptstyle{G/\wK}}$ since, if $q \in \calO$ and $k \in G$, we obtain:
\[
(g_{\scriptscriptstyle{G/\wK}})_{kq}({k_*E_i}_{\mid kq}, {k_*E_j}_{\mid kq})=
(g_{\scriptscriptstyle{G/\wK}})_{kq}(d\ltr_k(q).{E_i}_{\mid q}, d\ltr_k(q).{E_j}_{\mid q}) 
= (g_{\scriptscriptstyle{G/\wK}})_{k}({E_i}_{\mid q}, {E_j}_{\mid q}) = \delta_{ij}.
\]

\begin{lem}
  \label{lem2.2} \label{lem2.3}
  {\rm (1) } If
  $\vtheta \in \End T(G/\Kt)$ is $G$-invariant, the vector field
  $\delta\vtheta$ is $G$-invariant.
  \\
  {\rm (2)} If $V \in \calX(G/\Kt)^G$, the vector field
  $\nabla^*\nabla V$ is $G$-invariant.
  \\
  {\rm (3)} Let $\vtheta$ be $G$-invariant and
  $V \in \calX(G/\Kt)^G$.  If $\wk$ is generic,
  $\delta\vtheta \in \R \xi$ and $\nabla^*\nabla V \in \R \xi$.
\end{lem}

\begin{proof}
(1)  We need to show $\delta\vtheta = g_* \delta\vtheta$ for all $g \in G$.
Let $\calO$ and $(E_i)_i$ be as above.  We begin with the
 following observation.
 Let $X \in \calX(\calO)$ and $g\in G$, then the
vector field $\vtheta(X) : q \mto \vtheta_q(X_q)$ satisfies
\begin{equation}
  \label{eq2.5}
  g_* \vtheta(X) = \vtheta(g_*X) \ \; \text{on the open set $g\calO$}.
\end{equation}
Indeed, using the invariance of $\vtheta$ we have on $g\calO$:
\[
 g_* \vtheta(X)_{gq} = d\ltr_g(q).\vtheta_q(X)_q =
 \vtheta_{gq}(d\ltr_g(q).X_q) = \vtheta_{vq}(g_*X_{\mid gq}) =
 \vtheta(g_*X)_{gq}. 
\]
Therefore $g_*\vtheta(X) = \vtheta(g_*X)$ on $g\calO$.
Then,  by definition of
 $\delta\vtheta$, Remark~\ref{rema}, \eqref{eq2.1b} and~\eqref{eq2.5}, one gets
\[
 \delta\vtheta  =  \sum_{i}
  \nabla_{g_*E_i} \vtheta(g_*E_i) -   \vtheta(\nabla_{g_*E_i} g_*E_i)
 =  \sum_{i}g_*\nabla_{E_i}\vtheta(E_i)  - 
   g_*\vtheta(\nabla_{E_i} E_i) = g_* \delta\vtheta. 
\]
{\rm (2)} By Remark~\ref{rema} the vector field
$\nabla^*\nabla V$ on $g\calO$
  is equal to $-\sum_{i=1}^m
  \nabla^2_{g_*E_i,g_*E_i}V$. Thus, on $g\calO$ we have:
  \begin{align*}
    \nabla^*\nabla V & =   -\sum_i \nabla^2_{g_*E_i,g_*E_i}g_*V =
    -\sum_i \nabla_{g_*E_i}(\nabla_{g_*E_i}g_*V) -
    \nabla_{\nabla_{g_*E_i}g_*E_i}g_*V \ \; \text{(by invariance
    of $V$)}
    \\
   & = g_*\Bigl\{ -\sum_i \nabla_{E_i}(\nabla_{E_i}V) -
    \nabla_{\nabla_{E_i}E_i}V\Bigr\}  = g_* \nabla^*\nabla V \ \;
     \text{(by invariance of $\nabla$)}.
  \end{align*}
 {\rm (3)} follows from Definition~\ref{rem6.9} combined with  (1) and~(2).
\end{proof}

Recall that $\pi : G/\Kt \sto G/K$ is a principal
$\bbS^1$-bundle.  Fix a $G$-invariant metric
$g_{\scriptscriptstyle{G/K}}$ on $G/K$ given by scalars
$\kappa_\gamma >0$, $\gamma \in R_T^+$, and let $J$ be a
$G$-invariant almost complex structure on $G/K$ associated to an
invariant complex structure $J_\fm$ on $\fm$, see
Theorem~\ref{thm6.7}.

Set $\fs = \Lie(\bbS^1) = \fk/\wk = \R A$ with $A = [X_0 + \wk]$.
We begin by recalling a few (well-known) facts  about the
construction of $G$-invariant principal connections $\eta$ on the
principal bundle $G/\Kt$.\\
Such a  connection is an element of $\Omega^1(G/\Kt,\fs)$ such that
\begin{equation*} \label{eq4.1}
\eta_x(A^\#_x) = A, \quad \eta_{xk}(d\rtr_{k}(x).X) =
\Ad(k)^{-1}\eta_x(X) = \eta_x(X) \ \; \text{(since $\bbS^1$ is abelian)},
\end{equation*}
 for all $x \in G/\Kt$, $X \in T_x(G/\Kt)$ and $k \in \bbS^1$. Here,
 $\rtr_k$ is the right translation by $k \in \bbS^1$ and $A^\#$ is the
 vector field defined by $A^\#_{g\wpzero} = \frac{d}{dt}_{\mid t
   =0} (g e^{tX_0} \wpzero)$, hence $A^\# =\xi$
 (see~Definition~\ref{rem6.9}). 
The connection $\eta$ is said to be invariant if $\eta \in
\Omega(G/\Kt,\fs)^G$, i.e.,
\[
\eta_{gx}(d\ltr_g(x).X) = \eta_x(X) \ \; \text{for all $x \in
  G/\Kt$ and $g \in G$.}
\]
Using the fact that $\bbS^1$ is abelian, the invariant connections are described
by the following proposition, see, for example, \cite[pages 18 \&
39, Theorem~1.4.5]{CS}. Recall that, here,  the
exterior derivative of $\eta$ is given by
$d\eta(X,Y) = -\eta([X,Y])$ for $X,Y \in \calX(G/\Kt)$.

\begin{prop}
  \label{prop4.1}
  {\rm (1)} There exists a bijection  between the set of
  $G$-invariant principal connections and the set of linear maps $\nu : \g \to
  \fs$ such that:
  \begin{enumerate}[{\rm  \indent (a)}]
  \item  $\nu_ {\mid \fk}$ coincides with the canonical
    projection from $\fk$ to $\fs = \fk/\wk$;
    \item $\nu(\Ad(t)v) = \nu(v)$ for all $t \in K$ and $v \in \g$.
    \end{enumerate}
\noindent     {\rm (2)} the connection form of $\eta$ equals $d\eta \in
    \Omega^2(G/\Kt,\fs)^G$ and is determined by its value at
    $\wpzero$:
\[
d\eta_{g\wpzero}(U,V) = d\eta_\wpzero(d\ltr_{g^{-1}}(g\wpzero).U,d\ltr_{g^{-1}}(g\wpzero).V)
\]
for all $g \in G$ and $U,V \in T_{g\wpzero} (G/\Kt)$. 
\end{prop}

In the previous bijection, if $\nu : \g \to \fs$ one defines
$\eta_\wpzero$  by $\eta_\wpzero(\bar{u})= \nu(u)$ for all
$\bar{u} \in T_\wpzero (G/\Kt) \equiv \g/\wk$ (by condition~(a) this is well
defined).

 \begin{lem}
   \label{lem4.2}
  One defines a $G$-invariant principal connection $\eta$ on the principal
  bundle $\pi : G/\Kt  \sto  G/K$ by setting
\[
  \eta_{\wpzero}(U) = B(X_0,U) A \ \; \text{for all $U \in \g$.}
\]
Define $\Sigma \in \Omega^2(G/K,\fs)^G$ by $\Sigma_{p_0}(U,V) = 
\omega_{-X_0}(U,V) A$.  
 Then
 $d\eta  = \pi^*\Sigma$ and $\Sigma$ is $J$-invariant.
 \end{lem}

 \begin{proof}
   We first verify the conditions (a) and (b) of
   Proposition~\ref{prop4.1}. The orthogonal projection of
   $U \in \fk$ onto $\R X_0$ is equal to $B(X_0,U) X_0$,
   therefore its image in $\fs = \ft/\wt$ is
   $\eta_{\wpzero}(U) = B(X_0,U) A$, this gives~(a). Let
   $t \in K$ and $v \in \g$. Then we have:
$$
\eta_{\wpzero}(\Ad(t)v) = B(X_0,\Ad(t)v) A =
   B(\Ad(t)^{-1}X_0,v) A = B(X_0,v) A =
   \eta_{\wpzero}(v).
$$
   This proves~(b). Thus, Proposition~\ref{prop4.1} applies and $\eta$ is a
   $G$-invariant connection.
   \\
   Recall that $\omega_{X_0}(U,V) = B(X_0,[U,V])$ and that
   $\omega_{X_0}$ is $\Ad(K)$-invariant, see
   Proposition~\ref{prop6.6}. Thus $\Sigma_{p_0}$ identifies
   with an $\Ad(K)$-invariant element of
   $\Omega^2(\g/\fk,\fs)$. It follows that we can define a
   $G$-invariant form $\Sigma \in \Omega^2(G/K,\fs)^G$ by its
   value at $p_0$.  Then, since $\pi$, $d\eta$ and $\Sigma$
   are $G$-invariant it suffices to prove the second claim
   at the point $\wpzero$. Recall that
   $$d\pi({\wpzero}) : T_{\wpzero} (G/\Kt) \equiv \wm \to T_{p_0}
  ( G/K)\equiv \fm$$
  coincides with the projection from $\wm$
   onto $\fm$ given by $\wm = \R X_0 \obot \fm$. If
   $X,Y \in T_{\wpzero} (G/\Kt) \equiv \wm$, one has
 \[
d\eta_{\wpzero}(X,Y)= - \eta_{\wpzero}([X,Y]) = -B(X_0,[X,Y])A = -
B(X_0,[X_\fm,Y_\fm]) A,
\]
cf.~Proposition~\ref{prop6.6}. Therefore,
\[
d\eta_{\wpzero}(X,Y)= \Sigma_{p_0}(d\pi({\wpzero}).X,
d\pi({\wpzero}).Y) = (\pi^*\Sigma)_{\wpzero}(X,Y). 
\]
The $J$-invariance of $\Sigma$ follows from the
$J_\fm$-invariance of $\omega_{-X_0}$, cf.~Proposition~\ref{prop6.6}.
 \end{proof}

 \begin{Notation}
   \label{not4.3}
   In order to recover the notation used in \cite{Hat, Mor1,
     Mor2, Ogi} 
   we identify the one dimensional Lie algebra $\wt= \R A$ with
   $\R$ by sending $A$ to~$1$. Under this notation,
   $\eta \in \Omega^1(G/\Kt,\R)^G$, $\Sigma \in \Omega^2(G/K,\R)^G$,
   and:
   \begin{equation*}
     \label{eq4.3}
     \eta_\wpzero(X)= B(X_0,X), \quad \Sigma_\pzero(U,V) =
     -\omega_{X_0}(U,V) = - B(X_0,[U,V]),
   \end{equation*}
   for all $X \in \g$, $U,V \in \fm \equiv T_\pzero (G/K)$.
     Since $\eta_\wpzero(\xi_\wpzero) = B(X_0,X_0) = 1$
   we obtain  $\eta(\xi)= 1$ by $G$-invariance of~$\eta$.
 \end{Notation}

 The connection $\eta$ furnishes the decomposition $T(G/\Kt)= H(G/\Kt) \boplus V(G/\Kt)$ where
 $d\pi : H(G/\Kt) = \Ker \eta \isomto T(G/K)$ and $V(G/\Kt)= \Ker d\pi$ is
 a sub-bundle of rank one.
 \\
 Recall that we set 
 $\psi(p) = d\pi(p)^{-1} : T_{\pi(p)}(G/K) \isomto H_p(G/\Kt)$. One has
 $V_{g\wpzero}(G/\Kt)= d\ltr_g(\wpzero). V_\wpzero (G/\Kt)$ by invariance
 of $\pi$, and the invariance of $\eta$ implies that
 $H_{g\wpzero}(G/\Kt) = d\ltr_g(\wpzero). H_\wpzero (G/\Kt)$.
Clearly, $\xi$ is a $\calC^\infty$-generator of $V(G/\Kt)$.

 Recall from \cite{Hat} the construction of the
 endomorphism $\theta$, see~\eqref{eq3.1}; 
 for all $p \in G/\Kt$ and $X \in \calX(G/\Kt)$, we set:
\begin{equation}
  \label{eq4.4}
  \theta_p(X_p) = \psi(p)(Jd\pi(p).X_p).
\end{equation}

\begin{prop}
  \label{prop4.4}
  Let $\theta$ be as in~\eqref{eq4.4}. The
  endomorphism $\theta$ is $G$-invariant and the triple
  $\txe$ is a $G$-invariant \ACS on $G/\Kt$ such that $\pi \circ
  \theta = J \circ \pi$.
\end{prop}

\begin{proof}
Using \cite[\S 2, Theorem~1]{Hat} one obtains that $\txe$ is an
\ACS on $G/\Kt$ such that  $\pi \circ
  \theta = J \circ \pi$.
  By construction $\eta$ and $\xi$
  are $G$-invariant. Therefore it suffices to prove that
  $\theta$ is $G$-invariant.
 Notice first that $\psi$ is $G$-invariant. Indeed, for all
  $p \in G/\Kt$, $\psi({gp})$ is the inverse of $d\pi({gp}) = 
  d\ltr_g(\pi(p)) \circ d\pi(p) \circ d\ltr_{g^{-1}}(gp)$;
it follows that $\psi(gp)= d\ltr_g(p) \circ \psi(p) \circ
d\ltr_{g^{-1}}(g\pi(p))$ as required.   Then,
using the $G$-invariance of $J$, $\pi$ and $\psi$, 
one easily
gets $\theta_{gp}\circ d\ltr_g(p) = d\ltr_g(p) \circ \theta_p$. Thus $\theta$ is $G$-invariant.
\end{proof}

Recall that
$g_{\scriptscriptstyle{G/K}}$ is compatible with the almost complex structure~$J$,
i.e.~$(G/K,J,g_{\scriptscriptstyle{G/K}})$ is almost Hermitian,
cf.~Theorem~\ref{thm6.7}. Define a metric $g_{\scriptscriptstyle{G/\wK}}$ on $G/\Kt$
by
\begin{equation} \label{eq3.6a}
g_{\scriptscriptstyle{G/\wK}} = \pi^* g_{\scriptscriptstyle{G/K}} + \eta \otimes \eta,
\end{equation}
that is to say:
$(g_{\scriptscriptstyle{G/\wK}})_p(U,V) = g_{\scriptscriptstyle{G/K}}(d\pi(p).U,d\pi(p).V) + \eta(U)\eta(V)$
for all $U,V \in T_p(G/\Kt)$. 

\begin{lem}
  \label{lem4.5}
  The metric $g_{\scriptscriptstyle{G/\wK}}$ is $G$-invariant. The  $\wK$-invariant scalar product
  $g_\wm$ on $\wm= \R X_0 \boplus \fm$ defining $g_{\scriptscriptstyle{G/\wK}}$
  is given by $g_\wm = B$ on $\R X_0$ and
  $g_\wm = g_\fm$ on $\fm$. 
  \end{lem}

\begin{proof}
  The invariance of $g_{\scriptscriptstyle{G/\wK}}$ follows from the invariance
  of $d\pi$, $g_{\scriptscriptstyle{G/K}}$ and $\eta$. 
  Through the usual
  identifications, at the point $\wpzero$ one has
  \begin{align*}
    g_\wm(U,V) & = (g_{\scriptscriptstyle{G/\wK}})_\wpzero(B(X_0,U)X_0 + U_\fm,
  B(X_0,V)X_0 + V_\fm) = B(X_0,U)B(X_0,V) +
                 (g_{\scriptscriptstyle{G/K}})_{p_0}(U_\fm,V_\fm)
    \\
    &= \eta_\wpzero(U)\eta_\wpzero(V) +
  (\pi^*g_{\scriptscriptstyle{G/K}})_\wpzero(U,V).
\end{align*}
This proves the second claim.
 \end{proof}

We now combine the previous results to get the next
theorem, already  obtained in \cite[Theorem~4.1]{Cor} by
different methods. We have defined above an almost contact metric 
structure $\txe$ on $G/\Kt$ and a metric $g_{\scriptscriptstyle{G/\wK}}$.
Recall that if $J$ is a complex structure on $G/K$, it is
determined by the choice of an invariant ordering $\prootsQ$,
i.e.~by a Weyl chamber $\mathsf{C} \subset \fz_\R$,
cf.~Theorem~\ref{thm6.8}~(1).  Moreover, $(g_{\scriptscriptstyle{G/K}},J)$ is \kahler
\sissi the form $\Omega_\fm$ is equal to
$\omega_{ih_\kappa}$ for some $h_\kappa \in \mathsf{C}$,
see~Theorem~\ref{thm6.8}~(2). 

\begin{thm}
  \label{thm6.12}
  Let $J$ be a $G$-invariant almost complex structure and
  $g_{\scriptscriptstyle{G/K}}$ be a $G$-invariant almost Hermitian metric on $G/K$.
  \\
  {\rm (1)} Endow the principal bundle
  $\pi : M= G/\wK \sto  N= G/K$ with the \ACS $\txe$ and the
  metric $g_{\scriptscriptstyle{G/\wK}}$ 
  defined by~\eqref{eq3.6a}. Then $\txeMK$ is
  $G$-invariant and defines an \ACMS on~$G/\Kt$.
  \\
  {\rm (2)} Assume that $J$ is a complex structure on $G/K$. Then
  $\txeMK$ is a \NACMS on $G/\Kt$.
  \\
  {\rm (3)} The structure $\txeMK$ is c-Sasakian \sissi
  $(g_{\scriptscriptstyle{G/K}},J)$ is \kahler and
  $ch_\kappa = iX_0 = -Y_0\in \mathsf{C}$ for some $c >0$. In
  this case, $-Y_0 \in \mathsf{C}$ and
  $\gamma(Y_0) = -c\kappa_\gamma < 0$ for all
  $\gamma \in R_T^+= \varpi(\prootsQ)$.
\end{thm}

\begin{proof}
  (1) is consequence of Proposition~\ref{prop4.4},
  Lemma~\ref{lem4.5} and Theorem~\ref{thm1.1}.
  \\
  (2) Since $J$ is a complex structure, $g_{\scriptscriptstyle{G/K}}$ is
  Hermitian. Recall that $d\eta= \pi^*\Sigma$ by
  Lemma~\ref{lem4.2} and that $\Sigma$ is $J$-invariant,
  cf.~Lemma~\ref{lem4.2}. Then, Theorem~\ref{thm1.1}
  yields the result.
  \\
  (3)  Corollary~\ref{cor3.8} says that $\txeMK$ is c-Sasakian \sissi
  $(g_{\scriptscriptstyle{G/K}},J)$ is \kahler and $c\Omega^{g_{\scriptscriptstyle{G/K}}}= \Sigma$; the later
  is equivalent (by invariance and Lemma~\ref{lem4.2}) to 
  $c\omega_{ih_\kappa} = \omega_{-X_0}$. Therefore, by Proposition~\ref{prop6.6}~(3),
  $c\Omega^{g_{\scriptscriptstyle{G/K}}} = \Sigma$ \sissi $ch_\kappa = iX_0$.  Then,
  $iX_0 = -Y_0\in \Csf$ and if $\alpha \in \vpi^{-1}(\gamma)$ we get:
\begin{align*}
c\kappa_\alpha & = c \omega_{ih_\kappa}(X_\alpha,Y_\alpha) =
c\ascal{h_\kappa}{h_\alpha} = c\alpha(h_\kappa) \\ &=
\omega_{-iY_0}(X_\alpha,Y_\alpha) = -\ascal{Y_0}{h_\alpha} =
-\alpha(Y_0) = - \gamma(Y_0).
\end{align*}
 This proves the last assertion.
\end{proof}


Let $(g_{\scriptscriptstyle{G/K}},J)$ be a $G$-invariant almost Hermitian metric on
$G/K$ as in the previous section and let $\sigma= \txeMK$ be the
\ACMS obtained in Theorem~\ref{thm6.12}~(2).

\begin{rem}
  \label{rem4.6}
  Suppose that $(g_{\scriptscriptstyle{G/K}},J)$ is Hermitian.  Using
  Lemma~\ref{lem2.3}~(3) and Corollary~\ref{cor3.9} we can deduce
  the following result: assume that $\wk$ is generic, then
  $\sigma$ is a harmonic section/map \sissi $J$ is a harmonic
  section/map.
\end{rem}

The statement of the previous remark  is, in fact,   
true for any choice of $\wk$. In order to
prove  this  result we need to compute formulas involving  the
$G$-invariant Levi-Civita connection  $\nabla$ associated
to~$g_{\scriptscriptstyle{G/\wK}}$. 

We start by introducing more notation. Recall that $X_0=iY_0$
with $Y_0 \in \fz_\R$ and
$\fm_\gamma = \boplus_{\alpha \in \varpi^{-1}(\gamma) \cap
  \proots} \fm_\alpha$ for $\gamma \in R_T^+$. The
metric $g_{\scriptscriptstyle{G/\wK}}$ is determined by the scalar product
$g_\wm$ on $\wm$ as in Lemma~\ref{lem4.5} and the scalar product
$g_\fm$ on each $\fm_\gamma$ is equal to $\kappa_\gamma B$ for
some $\kappa_\gamma >0$, cf.~Theorem~\ref{thm6.7}.
For $\gamma \in R_T^+$ and  $\alpha \in
\vpi^{-1}(\gamma)$ we set:
\begin{equation*}
  \label{eq6.14}
  X'_\alpha = \frac{1}{\sqrt{\kappa_\gamma}} X_\alpha, \quad
  Y'_\alpha = \frac{1}{\sqrt{\kappa_\gamma}} Y_\alpha, \quad
  c_\gamma= \gamma(Y_0), \quad a_\gamma=
  \frac{c_\gamma}{\kappa_\gamma} = \frac{\gamma(Y_0)}{\kappa_\gamma}.
\end{equation*}
Then the elements $(X'_\alpha,  Y'_\alpha)$, $\alpha \in
\varpi^{-1}(\gamma)$, provide an orthonormal basis 
of $\fm_\alpha$ and
\[
\bigl\{X_0 \, ;  (X'_\alpha,  Y'_\alpha) \, : \, \alpha \in
\varpi^{-1}(\gamma) \cap \proots, \gamma \in R_T^+ \bigr\}
\]
is an orthonormal basis (for $g_\wm)$ of $\wm$. Easy computations yield, {for all
    $\lambda \in \proots \cap \vpi^{-1}(\gamma)$}:
\begin{gather*}
  B(ih_\lambda,X_0) = \ascal{h_\lambda}{Y_0} = c_\gamma, \quad
  (i h_\lambda)_\wm =  c_\gamma X_0,
  \\
  [X_0,X'_\lambda] = c_\gamma  Y'_\lambda, \quad  [X_0,Y'_\lambda]
  = - c_\gamma X'_\lambda
  \\
  [X'_\lambda,Y'_\lambda] = \frac{1}{\kappa_\gamma} i h_\lambda, \quad
  [X'_\lambda,Y'_\lambda]_\wm = a_\gamma X_0, \quad
    g_\wm(X_0, [X'_\lambda,Y'_\lambda]_\wm) = a_\gamma. 
\end{gather*}

Recall that we can find  $\mathcal{U}$ containing $e \in G$  such that $\wpi : \mathcal{U} \isomto \Ut= \wpi(\mathcal{U})$
is a diffeomorphism onto the open neighbourhood $\Ut$ of $\wpzero
\in G/\Kt$. If $X \in \wm \equiv T_\wpzero (G/\Kt)$ we defined a
vector field $X^* \in \calX(\Ut)$ by $X^*_{k\wpzero} =
d\ltr_k(\wpzero).X$ for $k \in \mathcal{U}$.
We know that the connection $\nabla$ is determined by
the $\Ad(\wK)$-invariant bilinear map
$\Lambda: \wm \times \wm \to \wm$ such that
$\Lambda(X,Y)= \half[X,Y]_{\wm} +U(X,Y)$, 
cf.~\eqref{eq2.2} and~\eqref{eq2.3}.
From the definition of $\xi$ we see that $X_0^* = \xi_{\mid \Ut}$
and, since $g_{\scriptscriptstyle{G/\wK}}$ is $G$-invariant, the set
\[
\bigl\{X_0^* \, ; 
{X'_\alpha}^*,  {Y'_\alpha}^* \bigr\}_{\{\alpha \in 
\varpi^{-1}(\gamma) \cap \proots, \gamma \in R_T^+\}}
\]
is an orthonormal frame on $\Ut$,  cf.~\cite[p.~51]{Nom}. To simplify the
notation we set
\[
F_\alpha= {X'_\alpha}^*, \quad \Phi_\alpha = {Y'_\alpha}^*, \quad
\text{for $\alpha \in \varpi^{-1}(\gamma) \cap \proots$.}
\]

\begin{lem}
  \label{lem6.15}
  Let $\gamma \in  R_T^+$ and
  $\alpha \in \varpi^{-1}(\gamma) \cap \proots$.  On the open subset $\Ut$,
  one has:
   \begin{align*}
     (0) & \ \ \nabla_\xi\xi = 0, \\
    (1) & \ \ \nabla_\xi F_\alpha  = \bigl(c_\gamma
          -\frac{a_\gamma}{2}\bigr) \Phi_\alpha, 
    \quad (2) \ \ \nabla_{F_\alpha} \xi = -\frac{a_\gamma}{2}
    \Phi_\alpha,
    \\  (3) & \ \ \nabla_\xi \Phi_\alpha  = \bigl(-c_\gamma +
         \frac{a_\gamma}{2}\bigr) F_\alpha, 
    \quad (4) \ \ \nabla_{\Phi_\alpha} \xi = \frac{a_\gamma}{2}
    F_\alpha,
    \\
   (5) & \ \ \nabla_{F_\alpha} \Phi_\alpha  = \frac{a_\gamma}{2}\xi,
    \quad (6)\ \ \nabla_{\Phi_\alpha} F_\alpha = -\frac{a_\gamma}{2}
    \xi,
    \\
     (7) & \ \ \nabla_{F_\alpha} F_\alpha  = 0,
    \quad (8) \ \ \nabla_{\Phi_\alpha} \Phi_\alpha = 0.
  \end{align*}
 \end{lem}

 \begin{proof}
  Since $U(X,Y)= U(Y,X)$, we have $\Lambda(Y,X) = \Lambda(X,Y)
  - [X,Y]_\wm$ and it suffices to check the formulas $(0), (1), (3),
  (5), (7), (8)$. Notice that to determine $U(X,Y)$ we need to compute the value
  of 
$$
f_{X,Y}(Z) = 2g_\wm(U(X,Y),Z) = g_\wm(X,[Z,Y]_\wm) + g_\wm(Y,[Z,X]_\wm)
   $$
   for $Z= X_0, X'_\lambda,Y'_\lambda$, $\lambda \in \proots$.
 The computation in the different cases is straightforward; we
 only give it in case~(1). Recall that
 $g_\wm(X_0,\wm) = g_\wm(\fm_\lambda,\fm_\mu) =0$ (for
 $\lambda \ne \mu$) and
 $[\fm_\alpha,\fm_\lambda] \subset \fm_{\alpha + \lambda} +
 \fm_{\alpha -\lambda}$ if $\alpha \neq \pm \lambda$,
 cf.~\eqref{eq1.6a}; this implies that if
 $V \in [\fm_\alpha,\fm_\lambda]$, we have
 $$g_\wm(X_0,V_\wm)= g_\wm(X_\alpha',V_\wm)=
 g_\wm(Y_\alpha',V_\wm) = 0.$$

  $(1)$ One easily sees that  $f_{X_0,X'_\alpha}(X_0)=
  f_{X_0,X'_\alpha}(X'_\alpha) = 0$. Let us give the details for
  the cases $Z= Y'_\alpha$ and  $Z=X'_\lambda$, $\lambda \neq \alpha$.
  One has
\begin{align*}
f_{X_0,X'_\alpha}(Y'_\alpha) & = g_\wm(X_0,[Y'_\alpha,X'_\alpha]_\wm) +
g_\wm(X'_\alpha,[Y'_\alpha,X_0]_\wm) \\ & = -a_\gamma +
  g_\wm(X'_\alpha,c_\gamma X'_\alpha) = -a_\gamma + c_\gamma. 
\end{align*}
If $Z = X'_\lambda$, $\lambda \neq \alpha$, we have
$[X'_\lambda,X'_\alpha] \subset \fm_{\lambda + \alpha} +
\fm_{\alpha - \lambda}$ and
$[X'_\lambda,X_0] \in \fm_\lambda$. This implies
$f_{X_0,X'_\alpha}(X'_\lambda)=0$.  A similar calculation shows
that $f_{X_0,X'_\alpha}(Y'_\lambda) = 0$. It follows that
$U(X_0,X'_\alpha) = \half (-a_\gamma + c_\gamma)
Y'_\alpha$. Since
$\half [X_0,X'_\alpha]_\wm =  \half c_\gamma Y'_\alpha$ we obtain
$\Lambda(X_0,X'_\alpha) = (c_\gamma - \half
a_\gamma)Y'_\alpha$. Consequently,
$\nabla_\xi F_\alpha = (c_\gamma - \half a_\gamma) \Phi_\alpha$.
\end{proof}

Recall that the almost complex structure $J$ on $G/K$ is given by
$J_\fm X_\lambda= \epsilon(\gamma) Y_\lambda$ and
$J_\fm(Y_\lambda)= - \epsilon(\gamma)X_\lambda$,
$\epsilon(\gamma)=\pm 1$, for all
$\lambda \in \vpi^{-1}(\gamma) \cap \proots$, see
Theorem~\ref{thm6.7}.

\begin{lem}
  \label{lem4.8}
  Let 
  $\alpha \in \rootsQ \cap \proots$ and $\gamma = \vpi(\alpha)$. Then,
  $\theta F_\alpha = \epsilon(\gamma) \Phi_\alpha$ and
  $\theta \Phi_\alpha= - \epsilon(\gamma) F_\alpha$ on~$\Ut$.
\end{lem}

\begin{proof}
  We make the calculation for $F_\alpha$, the same method will
  give $\theta \Phi_\alpha$.  Recall first that the projection
  $d\pi(\wpzero) : T_\wpzero (G/\Kt) \sto T_\pzero (G/K)$ can be
  identified with the orthogonal projection from $\wm$ onto
  $\fm$. Its inverse, $\psi(\wpzero)$, then identifies with the
  inclusion $\fm \ito \wm$. In other terms, the orthogonal
  decomposition (with respect to~$g_{\scriptscriptstyle{G/\wK}}$)
  $T_\wpzero (G/\Kt)  = H_\wpzero(G/\Kt) \boplus V_\wpzero (G/\Kt)$ identifies
  with the orthogonal decomposition $\wm = \fm \obot \R X_0$.
  \\
  The vector fields $F_\alpha$ and $\Phi_\alpha$ are determined
  by $(F_\alpha)_{\mid \wpzero} = X'_\alpha$ and
  $(\Phi_\alpha)_{\mid \wpzero} = Y'_\alpha$ in
  $\wm \equiv T_\wpzero (G/\Kt)$.  By the previous remarks,
  $d\pi(\wpzero).(F_\alpha)_{\mid \wpzero} = X'_\alpha$ and
  $\psi(\wpzero).Y'_\alpha = Y'_\alpha= (\Phi_\alpha)_{\mid
    \wpzero}$.  Let $k\in \mathcal{U}$. We get   $\theta_{k\wpzero} (F_\alpha)_{\mid k\wpzero} =
  \theta_{k\wpzero} d\ltr_k(\wpzero).(F_\alpha)_{\mid \wpzero} =
  d\ltr_k(\wpzero). \theta_{\wpzero}(F_\alpha)_{\mid \wpzero}$ by invariance of $\theta$
  and the definition of $F_\alpha$. 
  By   definition of $\theta$  we have 
$\theta_{\wpzero}(F_\alpha)_{\mid \wpzero} =
\psi(\wpzero)(J_{p_0} d\pi(\wpzero).(F_\alpha)_{\mid
  \wpzero})$.
Thus, through the previous identifications we obtain:
\begin{equation*}
\psi(\wpzero)(J_{p_0} d\pi(\wpzero).(F_\alpha)_{\mid
  \wpzero})  = \psi(\wpzero) (J_\pzero X'_\alpha)  = \epsilon(\gamma)
\psi(\wpzero)(Y'_\alpha) = \epsilon(\gamma) Y'_\alpha=
\epsilon(\gamma) (\Phi_\alpha)_{\mid
  \wpzero}. 
\end{equation*}
Hence,
$\theta_{k\wpzero} (F_\alpha)_{\mid k\wpzero} =  \epsilon(\gamma)
d\ltr_k(\wpzero).(\Phi_\alpha)_{\mid
  \wpzero} = \epsilon(\gamma) (\Phi_\alpha)_{\mid k\wpzero}$ and
we get $\theta F_\alpha
= \epsilon(\gamma) \Phi_\alpha$ on~$\Ut$.
\end{proof}

Observe that if two $G$-invariant vector fields coincide on
$\Ut$, they are equal on~$G/\Kt$. This remark and
Lemma~\ref{lem6.15}~$(0)$ imply that $\nabla_\xi\xi = 0$
(actually, $\xi$ is a Killing vector field). Furthermore, to
compute the invariant vector fields $\delta \theta$ and
$\rough \xi$ (see Lemma~\ref{lem2.3} and
Proposition~\ref{prop4.4}) it suffices to do it on $\Ut$.

\begin{thm}
  \label{prop6.16}
  Let $\txeMK$ be the \ACMS on $G/\Kt$ defined in
  Theorem~\ref{thm6.12}~(2). For $\gamma \in R_T^+$, set
  $n_\gamma= |\varpi^{-1}(\gamma) \cap \proots|$. 
   Under the previous notation we have:
  \begin{equation*}
\nabla_\xi \xi = 0, \quad    \delta \theta   = \Bigl(\sum_{\gamma \in R_T^+}
   \epsilon(\gamma) \, n_\gamma  \frac{\gamma(Y_0)}{\kappa_\gamma}\Bigl) \xi,
\quad
 \nabla^*\nabla  \xi   = \half\Bigl(\sum_{\gamma \in R_T^+}
 n_\gamma   \frac{\gamma(Y_0)^2}{\kappa_\gamma^2}\Bigr) \xi.
\end{equation*}
Assume that $J$ is a complex structure, then the \NACS  $\sigma=\txeMK$ is a harmonic section/map \sissi $J$ is a
  harmonic section/map.
 \end{thm}

 \begin{proof}
   Since $\theta\xi= \nabla_\xi\xi = 0$, the invariant vector fields
 $\delta \theta$ and $\nabla^*\nabla \xi$ have the following expressions
 on~$\Ut$:
\begin{gather*}
    \delta \theta   = \Bigl(\sum_{\alpha \in \proots \cap \Qsf}
    \nabla_{F_\alpha}\theta F_\alpha +
    \nabla_{\Phi_\alpha}\theta \Phi_\alpha \Bigr) + \nabla_\xi
    \theta \xi = \sum_{\gamma \in R_T^+} \, \sum_{\alpha \in
      \varpi^{-1}(\gamma) \cap \proots}
    \Bigl(\nabla_{F_\alpha}\theta F_\alpha +
    \nabla_{\Phi_\alpha}\theta \Phi_\alpha \Bigr), \\
 \nabla^*\nabla  \xi
     =  -\sum_{\gamma \in R_T^+} \, \sum_{\alpha \in
     \varpi^{-1}(\gamma) \cap \proots}
   \Bigl\{  \Bigl(\nabla_{F_\alpha}(\nabla_{F_\alpha}\xi)  - 
  \nabla_{\nabla_{F_\alpha}F_\alpha} \xi \Bigr) 
   + \Bigl(\nabla_{\Phi_\alpha}(\nabla_{\Phi_\alpha}\xi) -
  \nabla_{\nabla_{\Phi_\alpha}\Phi_\alpha} \xi\Bigr) \Bigr\}.
\end{gather*}
Let $\alpha \in \varpi^{-1}(\gamma) \cap \proots$. By
Lemma~\ref{lem4.8} and Lemma~\ref{lem6.15} we get
\[
\nabla_{F_\alpha}\theta F_\alpha +
    \nabla_{\Phi_\alpha}\theta \Phi_\alpha = \epsilon(\gamma)\Bigl(\nabla_{F_\alpha}\Phi_\alpha -
    \nabla_{\Phi_\alpha} F_\alpha\Bigr) = 
    \epsilon(\gamma)\Bigl(\half a_\gamma \xi  +\half a_\gamma \xi\Bigr) =
    \epsilon(\gamma) a_\gamma \xi.
\]
Since $a_\gamma = \gamma(Y_0)/\kappa_\gamma$ we have proved the
desired formula for $\delta \theta$.

The  computation of  $\nabla^*\nabla  \xi$ is similar. Indeed,
 \begin{align*}
 \phantom{,} &\Bigl(\nabla_{F_\alpha}(\nabla_{F_\alpha}\xi)  -
  \nabla_{\nabla_{F_\alpha}F_\alpha} \xi \Bigr) +
               \Bigl(\nabla_{\Phi_\alpha}(\nabla_{\Phi_\alpha}\xi)
               - \nabla_{\nabla_{\Phi_\alpha}\Phi_\alpha} \xi
               \Bigr) 
   \\
    & =
 \Bigl(\nabla_{F_\alpha} (-\frac{a_\gamma}{2}
  \Phi_\alpha) + \nabla_{\Phi_\alpha} (\frac{a_\gamma}{2}
  F_\alpha)\Bigr)  = 
  \Bigl(-\frac{a^2_\gamma}{4} \xi - \frac{a^2_\gamma}{4}\Bigr)
     \xi
= -\half   a_\gamma^2 \xi. 
\end{align*}
The formula for $\rough \xi$ follows. The last claim is a consequence of Corollary~\ref{cor3.9}.
 \end{proof}

 Assume that $J$ is a complex
 structure on $G/K$. Let $\sigma= \txeMK$ be the \NACMS on $G/\Kt$
 given by Theorem~\ref{thm6.12}.  
Recall that  conditions for $(g_{\scriptscriptstyle{G/K}},J)$ to be \kahler are given in
Theorem~\ref{thm6.8}, either in terms of the scalars
$\kappa_\gamma$ or of the KKS form defining the \kahler form
$\Omega^{g_{\scriptscriptstyle{G/K}}}$.

\begin{thm}
  \label{thm6.17}
 Assume that $(g_{\scriptscriptstyle{G/K}},J)$ is \kahler on $N=G/K$ and  endow $M=G/\Kt$ with the
 \NACMS $\sigma$.
  Then, this structure is a harmonic map.
\end{thm}

\begin{proof}
  We know that $J$ is a harmonic map when $(g_{\scriptscriptstyle{G/K}},J)$ is
  \kahler. Therefore the result is consequence of
  Theorem~\ref{prop6.16} and Corollary~\ref{cor3.9}.
\end{proof}

\begin{rems}
  \label{rem6.18}
  (0) When $\wk$ is generic, the proof of Theorem~\ref{thm6.17}
  is shorter since, thanks to Lemma~\ref{lem2.3}, we do not need
  Lemma~\ref{lem6.15} to apply Corollary~\ref{cor3.9}.
  
  Let $\mathsf{C}$ be the Weyl chamber defining the
  complex structure $J$.
  \\
\noindent   (1) It is not difficult to construct a harmonic \NACMS $\txeMK$
  on $G/\Kt$, as in Theorem~\ref{thm6.17}, which is not
  c-Sasakian.  By Theorem~\ref{thm6.12} it suffices to choose the
  \kahler metric $g_{\scriptscriptstyle{G/K}}$ such that, for all $c >0$ there exists
  $\gamma \in R_T^+$ such that $c\kappa_\gamma \neq
  \gamma(iX_0)$, that is to say choose  $h_\kappa \in
  \mathsf{C}$ such that $-Y_0 \notin \R_+ h_\kappa$.  (This is
  always the case if one starts with a choice of $\mathsf{C}$
  such that
   $iX_0=-Y_0 \notin \mathsf{C}$.)
  \\
  (2) Recall that $g_{\scriptscriptstyle{G/K}}$ is \kahler-Einstein when the form
  $\Omega_\fm$ is equal (up to a scalar multiple $c >0$)
  to $c\omega_{ih_\rho}$ with
  $\rho = \half\sum_{\alpha \in \prootsQ} \alpha$ (see
  Theorem~\ref{thm6.8}). In this case
  $k_\gamma = c\, \alpha(h_\rho)$ for all
  $\alpha \in \varpi^{-1}(\gamma)$. The structure $\txeMK$ is then
  Sasakian \sissi $iX_0=-Y_0= c \, h_\rho$. We can always
  produce a principal $\bbS^1$-bundle $\pi : G/\wK \sto G/K$ and a
  harmonic \NACMS $\txeMK$ on $G/\wK$ which is not Sasakian.
  Indeed, choose $\wz$ such that $\fz= \R X_0 \obot \wz$ with
  $-Y_0 \notin \R_+^* h_{\sigma}$ and define $\wk$ to be the
  orthogonal (with respect to~$B$) of $\R X_0$ in $\fk$.
  \\
  (3) In Theorem~\ref{thm6.17} we assume that $g_{\scriptscriptstyle{G/K}}$ is a
  \kahler metric in order to have the harmonicity of the
  section/map~$J$. A natural question is: what is the weakest
  hypothesis on $g_{\scriptscriptstyle{G/K}}$ to ensure that   $J$ is a
  harmonic section/map? One can observe that on a generalised flag manifold
$G/K$  equipped  with a complex structure $J$, then
\[
\text{$(g_{\scriptscriptstyle{G/K}},J)$ \kahler $\iff$ $(g_{\scriptscriptstyle{G/K}},J)$ quasi-\kahler
  (i.e.~$(1,2)$-symplectic) $\iff$
  $(g_{\scriptscriptstyle{G/K}},J)$ nearly-\kahler.} 
\]
This follows from \cite[Theorem~9.15 \&
Theorem~9.17]{WG}. Furthermore, every $G$-invariant almost
Hermitian metric on $G/K$ is semi-\kahler by
\cite[Theorem~8.9~(ii)]{WG}. 
\end{rems}

We now study the harmonicity of the Reeb vector field $\xi$.  Let
$(M,g_M)$ be a compact Riemannian manifold 
and let $\nabla$ be the associated Levi-Civita
connection. Denote by $UM$ the unit tangent bundle defined for
$x \in M$ by
$U_xM= \{ v \in T_xM \, : \,( g_M)_x(v,v) = 1\}$.  By
definition, a unit vector field $V$ on $M$ is a section of
$UM$ (thus it defines a vector field on $M)$.  We endow the
tangent bundle $TM$ with a g-natural metric as 
defined in \cite[5.1]{DP} (recall that the Sasaki metric is an
example of such a metric \cite[(5.6)]{DP}). The unit tangent
bundle is equipped with the restriction of this metric.  One
says that a unit vector field $V$ is a harmonic vector field,
resp.~a harmonic map, if $V$ is a harmonic as a section of
$UM$, resp.~as a map from $M$ to $UM$. Denote by
${R}(X,Y) 
= \nabla^2_{X,Y} - \nabla^2_{Y,X}$ the curvature tensor and set
\[
S(V) = \trace_{g_M} ({R}(\nabla_. V, V) .)=   
\sum_i{R}(\nabla_{E_i} V, V)E_i,
\]
where $(E_i)_i$ is an orthonormal frame for $g_M$. Combining
\cite[Theorem 2]{ACP}, \cite[Lemma 7.1 \& Proposition~8.1]{Str},
\cite[Corollary 2 \& Proposition~3]{ACP} and \cite[page 85]{HY},
we have the following characterisation of the harmonicity of $V$.

\begin{thm}
  \label{thm6.18} \label{thm6.18a}
{\rm (1)}  The unit vector field $V$ is a harmonic vector field \sissi
  $\nabla^*\nabla V= f V$ for some $f \in \calC^\infty(M,\R)$ and
  in this case one has $f= |\nabla V|^2$.
  \\
    {\rm (2)} Endow $UM$ with the Sasaki metric.  The unit vector
  field $V$ is a harmonic map \sissi $V$ is a harmonic vector
  field and $S(V)=0$. 
  \\
  {\rm (3)} Assume that the unit vector field $V$ is Killing.
  Then $V$ is a harmonic map for the Sasaki metric \sissi $V$ is a
  harmonic map for any g-natural metric.
\end{thm}

Suppose that $M=G/\wK$ is a homogeneous space ($\wK$ being
a closed connected subgroup of $G$) equipped with a $G$-invariant metric $g_{\scriptscriptstyle{G/\wK}}$.

\begin{lem}
  \label{lem6.19}
  Assume that $V$ is a $G$-invariant vector field on $G/\Kt$. Then,
  $S(V)$ is $G$-invariant.
\end{lem}

\begin{proof}
  Let $g \in G$.  The invariance of $\nabla$ implies
  $g_*\nabla^2_{X,Y}Z= \nabla^2_{g_*X,g_*Y} g_*Z$, thus
  $g_* {R}(X,Y)Z$ is equal to  $g_*{R}(g_*X,g_*Y)g_*Z$.
  Let $(E_i)_i$ be an orthonormal frame on an open subset $\calO$
  of~$G/\Kt$.
    Then, on $\calO$, we get by invariance of $V$:
  \[
S(V) =
{R}(\nabla_{g_*E_i} V, V)g_*E_i = {R}(\nabla_{g_*E_i} g_*V,g_* V)g_*E_i
= g_*. {R}(\nabla_{E_i} V, V)E_i = g_* S(V).
\]
This shows that $S(V)$ is an invariant vector field. 
\end{proof}

Now, assume that $\pi : M= G/\Kt \sto N= G/K$ is a principal
$\bbS^1$-bundle over a generalised flag manifold $G/K$ as in the
previous sections.
Endow $G/\Kt$ with the \ACMS $\txeMK$
defined in Theorem~\ref{thm6.12}~(1). By construction, the Reeb vector
field $\xi$ is a $G$-invariant unit vector field on $G/\Kt$.
  On a contact metric manifold the harmonicity of the Reeb vector field
  $\xi$, and is consequences, has been studied in details in
  \cite[\S7]{ACP}.

\begin{thm}
  \label{thm6.20}
  The map $\xi : G/\Kt \to U(G/\Kt)$ is a harmonic map for any
  g-natural metric on $T(G/\Kt)$.
\end{thm}

\begin{proof}
  We have shown in Theorem~\ref{prop6.16} that
  $\nabla^*\nabla \xi \in \R \xi$, therefore, by
  Theorem~\ref{thm6.18}~(1), $\xi$ is a harmonic unit vector
  field. Since $\xi$ is Killing, by Theorem~\ref{thm6.18a}~(2) \&~(3), it
  suffices to prove that $S(\xi)=0$. Furthermore, by invariance
  of $S(\xi)$ it is enough to show that $S(\xi)$ vanishes on the
  open subset $\Ut$ used in Lemma~\ref{lem6.15} and we can apply
  the formulas of that lemma. We calculate $S(\xi)$ in the
  orthonormal frame $\{\xi,
  (F_\alpha,\Phi_\alpha)_\alpha\}$. 
  Since $\nabla_\xi\xi = 0$ the vector field $S(\xi)$ is sum of
  terms of the form
  $${R}(\nabla_{F_\alpha}\xi,\xi)F_\alpha
  + {R}(\nabla_{\Phi_\alpha}\xi,\xi)\Phi_\alpha.$$  From
  the definition we obtain
\begin{align*}
{R}(\nabla_{F_\alpha}\xi,\xi)F_\alpha &  =
\nabla^2_{\nabla_{F_\alpha} \xi,\xi}F_\alpha -
  \nabla^2_{\xi,\nabla_{F_\alpha}\xi} F_\alpha  \\ &=
  \nabla_{\nabla_{F_\alpha} \xi}(\nabla_\xi F_\alpha) -
  \nabla_{\nabla_{\nabla_{F_\alpha} \xi} \xi}F_\alpha
  -\bigl\{\nabla_\xi(\nabla_{\nabla_{F_\alpha}\xi}F_\alpha)  -
  \nabla_{\nabla_\xi(\nabla_{F\alpha}\xi)} F_\alpha \bigr\}.
\end{align*}
Recall (with the notation of  Lemma~\ref{lem6.15}) that
$$\nabla_\xi {F_\alpha} = (c_\gamma - \frac{a_\gamma}{2}) 
\Phi_\alpha, \nabla_{F_\alpha} \xi = - \frac{a_\gamma}{2} 
\Phi_\alpha, \nabla_\xi {\Phi_\alpha} = \frac{a_\gamma}{2} 
F_\alpha \mbox{ and } \nabla_\xi {\Phi_\alpha} = - (c_\gamma - \frac{a_\gamma}{2}) 
F_\alpha.$$ Therefore, $\nabla_{F_\alpha}F_\alpha =0$ 
implies that:
\begin{align*}
{R}(\nabla_{F_\alpha}\xi,\xi)F_\alpha & =  \frac{a_\gamma}{2} 
   \nabla_{\nabla_{\Phi_\alpha} \xi}F_\alpha 
 + \frac{a_\gamma}{2}\bigl\{ \nabla_\xi(\nabla_{\Phi_\alpha}F_\alpha) -
  \nabla_{\nabla_\xi \Phi_\alpha} F_\alpha \bigr\} \\
 & =  \frac{a_\gamma}{2}\bigl\{ \nabla_\xi(- \frac{a_\gamma}{2}\xi) -
  \nabla_{ - (c_\gamma - \frac{a_\gamma}{2}) 
F_\alpha} F_\alpha \bigr\} = 0.
\end{align*}
A similar calculation gives
$R(\nabla_{\Phi_\alpha}\xi,\xi)\Phi_\alpha=0$. Thus $S(\xi)
= 0$ and $\xi$ is a harmonic map.
\end{proof}

\begin{Example}
Let $G= \SU(m+1)$, $K = S(\Unit(1)\times \Unit(m))$,  $\wK = \SU(m)
\subset K$.
Then the generalised flag manifold
$N=G/K = \SU(m+1)/S(\Unit(1)\times \Unit(m))$ is an irreducible  Hermitian
symmetric space. 
It follows that any $K$-invariant scalar product on $\fm$ is
proportional to $B$
and that there exist two almost complex structures on $G/K$, which
are complex structures.  We fix one of them, denoted by $J$. 
Any $G$-invariant \kahler metric is \kahler-Einstein and up to a
positive scalar there exists a ``unique'' $G$-invariant
\kahler-Einstein metric on $G/K$, that we denote by $g_{\scriptscriptstyle{G/K}}$. The
manifold $(G/K,g_{\scriptscriptstyle{G/K}},J)$ then identifies with $\C P^m$ equipped
with its standard structure.  Hence, since
$M=G/\Kt = \SU(m+1)/\SU(m)$ is isomorphic to the sphere
$\mathbb{S}^{2m+1}$, the fibration $G/\Kt \sto G/K$ is the Hopf
fibration $\pi : \mathbb{S}^{2m+1} \sto \C P^m$
as described in \cite[9.81]{Bes}. If we equip $G/\Kt$ with the structure
$\txeMK$, as above,  
we have shown that $\xi : G/\Kt \to U(G/\Kt)$ is a harmonic map
for any g-natural metric on $T(G/\Kt)$.
\end{Example}

Finally, we  illustrate the previous results 
by an example in the case where $N=G/T$ is the full flag
manifold, i.e.~$K=T$ is a maximal torus. In this setting the
notation simplifies since we have:
\[
\g= \ft \boplus \fm, \ \; \fm = \fn = \boplus_{\alpha \in
   \proots} \fm_\alpha,  \  \;  \fz= \ft, \ \;  
 \rootsP = \emptyset, \ \; \rootsQ= \roots, \ \;  \vpi =
 \id_{\ft_\R}, \  \; R_T= \roots.
\]
Therefore the irreducible $K$-modules $\fm_\gamma$ are the
irreducible $T$-modules
$\fm_\alpha = \R X_\alpha \boplus \R Y_\alpha$. The choice of an
invariant ordering is the choice of a set of positive roots
$\proots$, i.e.~the choice of a Weyl chamber $\Csf$ in $\ft_\R$.  The
$G$-invariant complex structure $J$ on $G/T$ associated to
$\proots$ is defined  by $J_\fm X_\alpha = Y_\alpha$,
$J_\fm Y_\alpha = - X_\alpha$ for $\alpha \in \proots$. Let
$g_{\scriptscriptstyle{G/T}}$ be a $G$-invariant metric on $G/T$, then 
the Euclidean product $g_\fm$ is given by $g_\fm =
\kappa_\alpha B$ on $\fm_\alpha$,  $\alpha \in \proots$. The metric $g_{\scriptscriptstyle{G/T}}$ is
\kahler \sissi $\kappa_\alpha + \kappa_\beta = \kappa_{\alpha +
  \beta}$ if $\alpha + \beta \in \proots$; equivalently, the
\kahler form $\Omega_\fm$ is equal to the KKS form
$\omega_{ih_\kappa}$ with $h_\kappa \in \Csf$.
Choose a sub-torus $\wT$ of $T$ such that
$T/\wT \cong \bbS^1$, let $X_0 =iY_0\in \ft$ (with
$Y_0 \in \ft_\R$) be such that $B(X_0,X_0) =1$ and
$\ft = \R X_0 \obot \wt$. Then
$\pi : G/\wT \sto  G/T$ is a principal $\bbS^1$-bundle.
Recall that a \NACMS $\txeMT$ on $G/\wT$ is constructed from: a
Hermitian metric $(g_{\scriptscriptstyle{G/T}},J)$, the connection $\eta$ such
that $\eta_\wpzero(U) = B(X_0,U)$, the metric
$g_{\scriptscriptstyle{G/\wT}} = \pi ^* g_{\scriptscriptstyle{G/T}} + \eta \otimes \eta$ and the
endomorphism $\theta$ lifting $J$ to the horizontal bundle.

We give an example of this construction in the case:
 \[
 G = \SU(3) \subset G_\C= \SL(3,\C), \quad \g= \su(3) \subset \g_\C=\fsl(3,\C).
 \]
 The split real form of $\g_\C$ is $\g_\R = \fsl(3,\R)$ and we
 choose the $G_\C$-invariant bilinear form
 $\ascal{A}{B} = \trace(AB)$ on $\g_\C$.  The split Cartan
 subalgebra $\ft_\R$ consists of the traceless diagonal real
 matrices, which will be written $\diag[x_1,x_2,x_3]$ with
 $x_1+x_2+x_3= 0$.  The compact Cartan subalgebra is
\[
\ft = i\ft_\R = \{i\diag[a,b,c] \, : \, \diag[a,b,c] \in \ft_\R\}.
\]
For $i=1,2,3$, let $\vepsilon_i \in \ft_\C^*$ be the linear form
which maps $\diag[x_1,x_2,x_3]$ to $x_i$. Set $\alpha_j=
\vepsilon_{j} - \vepsilon_{j+1}$ for $j=1,2$ and $\alpha_3 =
\alpha_1 + \alpha_2= \vepsilon_1 - \vepsilon_3$. Then 
one has 
$ \roots = \{\pm \alpha_1, \pm \alpha_2, \pm
  \alpha_3\} = \{\vepsilon_i - \vepsilon_j, \ 1 \le  i \neq j \le 3\}$.
  For
$\lambda = \vepsilon_k - \vepsilon_\ell$ the element $E_\lambda$
is the unit matrix $E_{k\ell}$.  Hence,
\begin{equation*} \label{eq5.0}
  X_\lambda = \frac{1}{\sqrt{2}} (E_{k\ell} - E_{\ell k}) , \quad
  Y_\lambda = \frac{i}{\sqrt{2}}(E_{k\ell} + E_{\ell k}).
\end{equation*}
The algebra $\g$ decomposes as the orthogonal direct sum $\g= \ft \obot \fm$. One has
$\fm= \fm_{\pm \alpha_1} \boplus \fm_{\pm \alpha_2} \boplus
\fm_{\pm \alpha_3}$, where
$\fm_{\pm \alpha_i} = \R X_{\pm \alpha_i} \boplus  \R Y_{\pm
  \alpha_i}$.
Any $T$-invariant  scalar product $g_\fm $ on $\fm$ is
defined by ${(g_\fm)}_{\mid \fm_{\alpha_j}} = \kappa_{\alpha_j} B$
with  $\kappa_{\alpha_j} >0$.
Finally, recall that a $T$-invariant complex structure $J_\fm$ on $\fm$
is determined by $J_\fm X_{\alpha_j} = \epsilon(\alpha_j) Y_{\alpha_j}$,
$J_\fm  Y_{\alpha_j} = - \epsilon(\alpha_j) X_{\alpha_j}$ with
$\epsilon(\alpha_j) =\pm 1$ for $j=1,2,3$.

We fix two integers $k,\ell \in \Z$ such that $(k,\ell) \neq (0,0)$ and set  $\Gamma = \sqrt{k^2
  + \ell^2 + k\ell}$.
Define an  element of~$\ft$ by $X_{-1} =
\frac{i}{\Gamma\sqrt{2}} \diag[k,\ell, -(k +\ell)]$
and set
\[
\wt= \ft_{k,\ell} = \R X_{-1}, \quad \wT = T_{k,\ell} =
  \exp(\wt) \subset T.
\]
Therefore, $\wT$ is the one-dimensional sub-torus of $T$ such that
\[
\wT= T_{k,\ell} = \{\diag[e^{itk}, e^{it\ell}, e^{-it(k+\ell)}] \, : \, t \in \R\}.
\]
Let $X_0 \in \ft $  such that $B(X_{-1},X_0) = 0$, $B(X_0,X_0) = 1$:
\begin{equation*}
  \label{eq5.4}
  X_0=
\frac{-i}{\Gamma\sqrt{6}} \diag[2\ell +k, -(\ell + 2k) , -\ell
+k] = iY_0, \quad  Y_0= 
\frac{-1}{\Gamma\sqrt{6}} \diag[2\ell +k, -(\ell + 2k) , -\ell
+k].
\end{equation*}
Notice that $(X_{-1},X_0)$ is an orthonormal basis, for $B$, of
$\ft = \ftkl \boplus \R X_0= \R X_{-1} \boplus \R X_0$.
The group  $\bbS^1= T/\wT$ is a one-dimensional torus
such that $\Lie(\bbS^1) = \fs = \ft/\wt =\R A$, with $A = [X_0 + \wt]$.
One can therefore define a principal $\bbS^1$-bundle
\[
\pi :  M=G/\Tkl \sto N =G/T.
\]
We say that this homogeneous  bundle is  an \emph{Aloff-Wallach 
  space} \cite{AW}. 
Recall that we set
$\wm= \R X_0 \boplus \fm = \R X_0 \boplus \fm_{\alpha_1} \boplus
\fm_{\alpha_2} \boplus \fm_{\alpha_3}$,
and that $\wm$ is a $\wT$-module under the adjoint action.
One easily sees, by computing
$\lambda(X_{-1})$ for $\lambda \in \roots$, that:
\begin{equation*} \label{eq5.2}
i\wt =\Ker(\alpha_1) \iff k =\ell, \  i\wt =\Ker(\alpha_2) \iff
2\ell +k=0, \  i\wt =\Ker(\alpha_3) \iff 2k +\ell =0.
\end{equation*}
Then, $\wt$ is generic if and only if $k \neq \ell$, $2\ell +
 k \neq 0$, $2k+\ell \neq 0$.  Proposition~\ref{prop6.9}
 becomes in this example: 

\begin{prop}
  \label{prop5.3}
   The space of $G$-invariant vector fields
   $\calX(G/T_{k,\ell})^G$ is equal to:
   \begin{enumerate}[{\rm \indent (1)}]
     \item $\R \xi$ when $k \neq \ell$, $2\ell +
 k \neq 0$, $2k+\ell \neq 0$; 
\item  $\R \xi \boplus \R
 X_{\alpha_1}^\# \boplus Y_{\alpha_1}^\#$ if $k=\ell$;
 \item  $\R \xi \boplus \R
 X_{\alpha_2}^\# \boplus Y_{\alpha_2}^\#$ if $2\ell +k =0$;
   \item $\R \xi \boplus \R
  X_{\alpha_3}^\# \boplus Y_{\alpha_3}^\#$ if $2k +\ell=0$.
  \end{enumerate}
\end{prop}

To give an example, we now choose the complex structure $J$ on $G/T$ associated to the
basis $\broots=\{\alpha = \alpha_3, \, \beta= -\alpha_2\}$ of
$\roots$, hence
$\proots=\{\alpha = \alpha_3, \, \beta= -\alpha_2, \, \gamma =
  \alpha+\beta= \alpha_1\}$
and the corresponding Weyl chamber is $\Csf =\{\diag[a,b,c] \in
\ft_\R : a > c > b\}$. 
Then, we have 
\[
  (X_\alpha = X_{\alpha_3},  \,Y_\alpha =Y_{\alpha_3}), \ (X_\beta = -X_{\alpha_2}, \,
  Y_\beta =Y_{\alpha_2}), \ (X_\gamma = X_{\alpha_1}, \, Y_\gamma
  = Y_{\alpha_1}),
\]
and $J_\fm X_\lambda= Y_\lambda$, $J_\fm Y_\lambda= -X_{\lambda}$
for $\lambda =\alpha,\beta,\gamma$. 
Moreover, under the choice of $Y_0$ made above:
\begin{equation*}
  \label{eq5.9}
  \alpha(Y_0) = -\frac{\sqrt{3}}{\Gamma\sqrt{2}}  \ell, \quad
  \beta(Y_0) = -\frac{\sqrt{3}}{\Gamma\sqrt{2}} k, \quad \gamma(Y_0)=
  -\frac{\sqrt{3}}{\Gamma\sqrt{2}} (\ell +k).
\end{equation*}
Thus, $-Y_0 \in \Csf$ \sissi $k >0$ and $\ell >0$.
Notice that $\fm_\alpha= \fm_{\alpha_3}$,
$\fm_\beta = \fm_{\alpha_2}$ and $\fm_\gamma =
\fm_{\alpha_1}$. Let $(g_{\scriptscriptstyle{G/T}},J)$ be a Hermitian metric on $G/T$
determined by a $T$-invariant Euclidean product $g_\fm$ as
above, i.e.~by a triple
$(\kappa_\alpha,\kappa_\beta,\kappa_\gamma)$ of positive
scalars. Then, $(g_{\scriptscriptstyle{G/T}},J)$ is \kahler \sissi
$\kappa_\alpha + \kappa_\beta = \kappa_\gamma$; equivalently, for all $U,V \in \fm$,
$g_\fm(U,V) = \omega_{i h_\kappa}(U, J_\fm V)$
with $h_\kappa \in \Csf$ and in this case
$\kappa_\alpha= \alpha(h_\kappa), \kappa_\beta = \beta(h_\kappa)$
(thus $\kappa_\gamma= \gamma(h_\kappa)$).

Recall that from $J,X_0,g_{\scriptscriptstyle{G/T}}$ we can construct an \ACMS
$\sigma = \txeMT$ on
$G/\wT=G/T_{k,\ell}$.  Theorem~\ref{thm6.17} yields:

\begin{thm}
  \label{thm4.1a}
  Let $J$ be as above and choose a \kahler metric $(g_{\scriptscriptstyle{G/T}},J)$ on
  $G/T$. Then, the \NACMS structure $\sigma$ on $G/T_{k,\ell}$ is  a harmonic map.
\end{thm}

  By Theorem~\ref{thm6.12}, $\sigma$ is c-Sasakian \sissi there
  exists $c >0$ such that $- Y_0 = c h_\kappa$,
  i.e.~$c \kappa_\alpha= \frac{\sqrt{3}}{\Gamma \sqrt{2}} \ell$
  and $c\kappa_\beta= \frac{\sqrt{3}}{\Gamma \sqrt{2}}
  k$. Therefore, it suffices to pick $h_\kappa \in \Csf$ such
  that $h_\kappa \notin -\R_+^* Y_0$ to get that $\sigma$ is not
  c-Sasakian.  In particular, if $k$ or $\ell$ is $<0$, the
  structure $\sigma$ is not c-Sasakian. Moreover, the metric
  $g_{\scriptscriptstyle{G/T}}$ is \kahler-Einstein \sissi $h_{\kappa}$ is a multiple
  of $h_\rho = h_\gamma = \diag[1,-1,0]$. Therefore, for this
  choice of $g_{\scriptscriptstyle{G/T}}$, the structure $\sigma$ is c-Sasakian when
  $-Y_0 = c\diag[1,-1,0]$, which is equivalent to $k =\ell$ and
  $\frac{\sqrt{3}}{\Gamma\sqrt{2}} \ell = c >0$. Thus, if
  $k \neq \ell$ or $k =\ell < 0$ the structure $\sigma$ is not
  Sasakian with $g_{\scriptscriptstyle{G/T}}$ \kahler-Einstein.

For completeness, we give the expression of the vector fields
$\delta \theta$, $\nabla^* \nabla \xi$ obtained in
Theorem~\ref{prop6.16}:
\[
  \delta \theta= -\frac{\sqrt{3}}{\Gamma\sqrt{2}}
  \Bigl(\frac{\ell}{\kappa_\alpha} + \frac{k}{\kappa_\beta} +
  \frac{(\ell +k)}{\kappa_\gamma}\Bigr)\xi, \quad
  \nabla^*\nabla \xi = \frac{3}{4\Gamma^2}
  \Bigl(\frac{\ell^2}{\kappa_\alpha^2} +
  \frac{k^2}{\kappa_\beta^2} + \frac{(\ell
    +k)^2}{\kappa_\gamma^2}\Bigr)\xi.
\]


\vfill


\begin{thebibliography}{AAAA}

  \bibitem[ACP]{ACP} M.T.K.~Abassi, G.~Calvaruso, and D.~Perrone,
  Harmonicity of unit vector fields with respect to Riemannian
  g-natural metrics, \emph{Differential Geometry and its
    Applications}, \textbf{27} (2009), 157--169.


    \bibitem[AW]{AW} S.~Aloff and N.~Wallach, An infinite family of
distinct $7$-manifolds admitting positively curved Riemannian
structures, \emph{Bull. Amer. Math. Soc}, \textbf{81} (1975),
93--97.


   \bibitem[Arv]{Arv} A.~Arvanitoyeorgos, \emph{An Introduction
to Lie Groups and the Geometry of Homogeneous Spaces}, Student
Math.Library, \textbf{22}, Amer. Math. Soc., 2003.



\bibitem[BW]{Baird-Wood}
P.~Baird and J.C.~Wood,
\newblock{\em Harmonic morphisms between Riemannian manifolds},
\newblock London Math. Soc. Monogr. No. {\bf{29}}, Oxford
University Press, 2003.

 \bibitem[Bes]{Bes} A.L.~Besse, \emph{Einstein
 Manifolds}, Classics in Mathematics, Springer Verlag, 2008.
    

\bibitem[Bla]{Blair}
D.E.~Blair,
\newblock {\em Riemannian geometry of contact and symplectic manifolds}, 
\newblock Progress in Mathematics, vol. 203, Birk\"auser, 2002.


     
   \bibitem[Bou]{Bou} N.~Bourbaki, {\it Groupes et alg\`ebres
de Lie, Chapitres~7, 8 et 9}, Masson, Paris, 1981.

\bibitem[CS]{CS} A.~\v{C}ap and J.~Slovak, \emph{Parabolic
Geometries~I}, Mathematical Surveys and Monographs, \textbf{154},
AMS, 2009.

\bibitem[Cor]{Cor} E.M.~Correa, Hermitian non-K\"ahler
tructures on products of principal $S^1$-bundles over complex
flag manifolds and applications in Hermitian geometry with
torsion, \emph{Asian J.  Math.}, \textbf{28} (2024),
No.~3, 437-500; 
\url{https://arxiv.org/abs/1803.09170}.

 \bibitem[DP]{DP} S.~Dragomir and D.~Perrone, \emph{Harmonic
      Vector Fields: Variational Principles and Differential
      Geometry}, Elsevier, 2012.

    


\bibitem[Gra]{Gray}
J.~Gray, 
\newblock Some global properties of contact structures, 
\newblock {\em Ann. of Math.}, {\bf 69} (1959), 421--450. 

\bibitem[Gue]{Guest}
M.A.~Guest, 
\newblock Geometry of maps between generalized flag manifolds, 
\newblock {\em J. Diff. Geom.}, {\bf 25} (1987), 223--247. 

\bibitem[HY]{HY} S.D.~Han and J.W.~Yim, Unit vector fields
      on spheres which are harmonic maps, \emph{Math.~Z.},
      \textbf{227} (1998), 83--92.
  
  \bibitem[Hat]{Hat} Y.~Hatakeyama, Some Notes on
Differentiable Manifolds with Almost Contact Structures,
\emph{T\^ohoku Math. Journal}, \textbf{15} (1963), 176--181.

\bibitem[KN]{KN} S.~Kobayashi and K.~Nomizu, \emph{Foundations of
Differential Geometry~II}, John Wiley and Sons, 1969.

  
\bibitem[LV]{LV}
E.~Loubeau and E.~Vergara-Diaz, 
\newblock The harmonicity of nearly cosymplectic structures, 
\newblock \textit{Trans. Amer. Math. Soc.}, {\bf 367} (2015), 5301--5327. 


  \bibitem[Mor1]{Mor1} A.~Morimoto, On Normal Almost Contact
Structures, \emph{J. Math.Soc. Japan}, \textbf{15} No~4 (1963),
420--436.

    \bibitem[Mor2]{Mor2} \bysame, On Normal Almost Contact
Structures with a Regularity, \emph{T\^ohoku Math. Journal},
\textbf{16}~(1) (1964), 90--104.

 \bibitem[Nom]{Nom} K.~Nomizu, Invariant Affine Connections on
Homogeneous Spaces, \emph{Amer.~J. Math}, \textbf{76} No~1
(1954), 33--65.

    \bibitem[Ogi]{Ogi} K.~Ogiue, On Fiberings of Almost Contact
Manifolds, \emph{Kodai Math. Sem. Rep.}, \textbf{17} No~1 (1965),
53--62.


\bibitem[One]{ONeill}
B.~O'Neill,
\newblock The fundamental equations of a submersion, 
\newblock {\em Michigan Math. J.}, {\bf 13} (1966), 459-470.

\bibitem[Per]{Perrone}
D.~Perrone, 
\newblock Contact metric manifolds whose characteristic vector field is a harmonic vector field, 
\newblock {\em Diff. Geom. Appl.}, {\bf 20} (2004), 367--378.

\bibitem[Str]{Str} M.A.W.~Strachan, Harmonic Vector Fields on
  Riemannian Manifolds, \emph{Msc by research thesis, University
    of York}, {2014}, 
  \url{https://etheses.whiterose.ac.uk/9212/}.
  
\bibitem[Ver]{thesis}
E.~Vergara-Diaz,
\newblock Harmonic sections and almost contact structures, 
\newblock \emph{Ph.D. thesis, University of York}, 2006.

\bibitem[VW1]{VW1}
E.~Vergara-Diaz and C.M.~Wood,
\newblock Harmonic almost contact structures, 
\newblock {\em Geom. Dedicata}, {\bf 123} (2006), 131-151.

\bibitem[VW2]{VW2}
\bysame,
\newblock Harmonic contact metric structures and submersions, 
\newblock {\em Int. J. Math.}, {\bf 20} (2009), 209-225.

\bibitem[Wie]{Wiegmink} G.~Wiegmink, Total bending of vector fields on Riemannian manifolds, \emph{Math. Ann.}, \textbf{303} (1995), 325--344.

\bibitem[WG]{WG} J.A.~Wolf and A.~Gray, Homogeneous Spaces
Defined by Lie Group Automorphisms, I \&~II, \emph{J.~of Differential
Geometry}, \textbf{2} (1968), 77--114 \& 115--159.


\bibitem[Woo]{CMW2}
C.M.~Wood,
\newblock Harmonic almost complex structures, 
\newblock {\em Compositio Mathematica}, {\bf 99} (1995), 183-211.


\end{thebibliography}
\end{document}